\documentclass{amsart}

\usepackage{amsmath,amsthm,amscd,amssymb}
\usepackage{tikz}
\usepackage{mathrsfs}
\usepackage{comment}
\usepackage{graphicx}
\usepackage{enumerate}

\usepackage{color}

\usepackage{amsmath, amssymb,color,amsthm,graphicx}
\newtheorem{main}{Theorem}
\theoremstyle{plain}
\newtheorem{thm}{Theorem}[section]
\newtheorem{cor}[thm]{Corollary}
\newtheorem{lem}[thm]{Lemma}
\newtheorem{clm}[thm]{Claim}

\theoremstyle{definition}
\newtheorem{dfn}[thm]{Definition}

\theoremstyle{remark}
\newtheorem{remark}[thm]{Remark}

\numberwithin{equation}{section}

\usepackage{bm}
\usepackage{float}

\title[Non-existence of Lyapunov exponents in the Newhouse domain]{Non-existence of Lyapunov exponents in the Newhouse domain}

\author{Shin Kiriki}
\address[Shin Kiriki]{Department of Mathematics, Tokai 
University, 4-1-1 Kitakaname, Hiratsuka, Kanagawa, 259-1292, JAPAN}
\email{kiriki@tokai.ac.jp}

\author{Xiaolong Li}
\address[Xiaolong Li]{School of Mathematics and Statistics, 
Huazhong University of Science and Technology, Luoyu Road 1037, Wuhan, 430074, CHINA}
\email{lixl@hust.edu.cn}

\author{Yushi Nakano}
\address[Yushi Nakano]{
Department of Mathematics, Hokkaido University, Kita 10, Nishi 8, Kita-Ku, Sapporo, Hokkaido, 060-0810, 
JAPAN}
\email{yushi.nakano@math.sci.hokudai.ac.jp}

\author{Teruhiko Soma}
\address[Teruhiko Soma]{Department of Mathematical Sciences, 
Tokyo Metropolitan University, 1-1 Minami-Osawa, Hachioji, Tokyo, 192-0397, JAPAN}
\email{tsoma@tmu.ac.jp}

\subjclass[2020]{37C05, 37C40, 37C29}

\keywords{Lyapunov exponent; homoclinic tangency; Newhouse domain; non-hyperbolic dynamics; wandering domain}

\begin{document}

\begin{abstract}
We show that within the Newhouse domain of $C^r$ surface diffeomorphisms ($r \in [2,\infty )$), there exists a dense subset $\mathcal{D}$ such that for any $f \in \mathcal{D}$, Lyapunov exponents fail to exist for all points in some open set $U$ and all nonzero tangent vectors in some open cone $V \subset \mathbb{R}^2$.
This demonstrates that the non-existence of Lyapunov exponents is a persistent phenomenon in the setting of robust homoclinic tangencies.
The proof relies on constructing diffeomorphisms exhibiting specific oscillatory return times near a homoclinic tangency, incorporating techniques from Newhouse theory and recent results on Lyapunov irregularity, alongside several refinements and new arguments.
\end{abstract}

\maketitle

\section{Introduction}

Dynamical systems with a robust homoclinic tangency -- those in the so-called Newhouse domain -- are a major source of extremely non-hyperbolic dynamics and have long been known to exhibit a rich variety of complex dynamical phenomena, 
such as infinitely many sinks \cite{N1974}, infinitely many strange attractors \cite{Colli1998}, super-exponential growth of the number of periodic points \cite{Kaloshin2000}, universal dynamics \cite{Turaev2015}, historic wandering domains \cite{KS2017,KNS2022} and pluripotency \cite{KLNSV2025}.
This paper focuses on statistical irregularities associated with one of the most fundamental quantities in dynamical systems: the Lyapunov exponent, which quantifies the sensitivity of an orbit to initial conditions and is widely used as an indicator of chaos in the natural sciences.
We show that, within the Newhouse domain, Lyapunov exponents persistently fail to exist on open sets of points and tangent vectors:

\begin{main}\label{thm:main}
Given a closed surface $M$ and $r\in [2,\infty )$, there is a dense subset $\mathcal D$ of the Newhouse domain of 
$\mathrm{Diff}^r(M)$ 
such that for any $f\in \mathcal D$, there exist nonempty open sets $U\subset M$ and  $V\subset \mathbb R^2$, under the identification of $TU$ with $U \times \mathbb R^2$, such that the Lyapunov exponent
\[
\text{$\displaystyle \lim_{n\to+\infty} \frac{\log ||Df^n(x)v||}{n}$ does not exist}
\]
for all $x\in U$ and nonzero $v\in V$.
\end{main}

In light of the widely believed dichotomy suggesting that `generic' surface diffeomorphisms are either 
hyperbolic
or 
contain homoclinic tangencies, together with its higher-dimensional counterpart involving dynamics with heterodimensional cycles \cite{PS2000,BDP2003,P2005}, our result provides significant insight into the statistical properties of non-hyperbolic dynamics. It highlights that complex behavior regarding Lyapunov exponents, specifically their non-existence, widely occurs within these non-hyperbolic systems, potentially contributing to a more complete understanding of dynamics beyond uniform hyperbolicity.

Despite the fundamental importance of Lyapunov exponents in characterizing dynamical behavior and their applications in various scientific fields (cf.~\cite{Viana2014}), the phenomenon where these exponents fail to converge for observable points or vectors (i.e.~those in a positive Lebesgue measure set) has remained relatively less explored.
Early work by Ott and Yorke \cite{OY2008} presented the first example of such non-convergence in a surface flow with an attracting homoclinic loop.
More recently, our previous work \cite{KLNS2022} aimed to uncover the underlying mechanisms driving the non-existence of Lyapunov exponents
 for certain surface diffeomorphisms.
This mechanism was subsequently adapted to demonstrate the non-existence of Lyapunov exponents in a class of two-dimensional piecewise expanding maps \cite{NSY2023}.

This paper builds upon and integrates the techniques developed in the previous work \cite{KLNS2022} with the rich technical machinery accumulated since the era of Newhouse, Palis and Takens \cite{PT1995} for studying the Newhouse domain.
 In particular, it leverages the methods for constructing diffeomorphisms with prescribed return time properties of wandering domains developed by Kiriki and Soma~\cite{KS2017}, which were themselves inspired by earlier work of Colli and Vargas~\cite{CV2001}.
 Furthermore, we introduce several new arguments and refinements.
Specifically, 
accurately estimating certain newly arising off-diagonal terms (cf.~the (1,2)-entry of \eqref{eq:0412}) in the derivative matrices presents significant challenges.
This necessitated modifications to the arguments, the development of a new `Modified Rectangle Lemma' (Section \ref{Sec3}), an inductive analysis of the products of Jacobian matrices for the lower bounds needed to control the relevant coefficients despite the presence of newly arising off-diagonal terms (Section \ref{s:6}),
and a fine-tuning of the arguments both in \cite{KS2017} and \cite{KLNS2022} to establish the required properties, including specific conditions on 
strong dissipativity, and exponential oscillations of return times (conditions detailed in Section \ref{Sec2}).

Another observation in our construction suggests that the mechanism leading to the non-existence of Lyapunov exponents identified here operates independently of Birkhoff irregularity (non-existence of time averages).
Indeed, one can show the convergence of time averages for any continuous observable at every point within the open set $U$ of Theorem \ref{thm:main} on which the Lyapunov exponent does not exist (see Remark \ref{rmk:1012}), emphasizing that these two quantities are independent.

We close this introduction by presenting an application of Theorem \ref{thm:main} to the H\'enon family, 
i.e., diffeomorphisms on $\mathbb{R}^2$ given by $f_{a,b}(x, y)=(1 - ax^2 + by, x)$ with $(a,b)\in\mathbb R^2$, 
 in sharp contrast to pioneering works \cite{BC91, MV93, BY93, BV01} on the existence and positivity of Lyapunov exponents of the H\'enon family.
The following is an immediate consequence of Theorem \ref{thm:main}.
\enlargethispage{\baselineskip}
\begin{cor}
There is an open set $\mathcal{O}$ in $\mathbb R^2$ 
with 
$(2, 0) \in\overline{\mathcal O}$ such that, for every 
$(a,b) \in\mathcal O$, 
the H\'enon map $f_{a,b}$ can be $C^{r}$-approximated by diffeomorphisms with a nonempty open set
on which Lyapunov exponents do not exist.
\end{cor}

\section{Non-existence of Lyapunov exponents in abstract settings}\label{Sec2}

Throughout this paper, let $r$ be a fixed real number with $r\ge 2$, $M$ a $C^r$ closed surface and $\mbox{Diff}^r(M)$ the space of all $C^r$ diffeomorphisms of $M$ endowed with the $C^r$ topology. 
A guiding principle of this paper is that one important source of oscillation of Lyapunov exponents is the interaction between hyperbolic dynamics and tangential returns. The former produces strong anisotropic growth, while the latter reconfigures the expanding and contracting directions. The following conditions are designed to capture this hybrid mechanism in an explicit and quantitative form.

\begin{dfn}
We say that $f\in \mathrm{Diff}^r(M)$ satisfies {\bf DH}, {\bf H}, {\bf QT} and {\bf OE} if there are two sequences $(n_k^H)_{k\ge 1}$, $( n_k^T)_{k\ge 1}$ of positive integers with $\lim\limits_{k\to\infty}n_k^H
=+\infty$ and a sequence of points $(\zeta_k)_{k\ge 1}$ of $M$ such that the following conditions hold.
\begin{itemize}
\item[{\bf DH}] (Dominance of hyperbolic return times) 
For each $k\in \mathbb N$, with the notation $n_k=n_k^H+n_k^T$, we have 
\[
f^{n_k}(\zeta_k)  = \zeta_{k+1}\quad \text{and} \quad \lim_{k\to \infty} \frac{n_k^H}{n_k} =1.
\]

\item[{\bf H}] (Strongly dissipative hyperbolic region) 
There exist 
an open set $K\subset M$ 
and a $C^r$ coordinate on $K$ in which
\[
f(x,y) = (\lambda_s x,\, \lambda_u y)
\]
with constants $0 < \lambda_s < 1 < \lambda_u$ satisfying 
\[
\lambda_s \lambda_u^3 < 1.
\]
Furthermore, there exists a nonempty closed set $K'\subset K$ such that
\[
f^j(\zeta_k)\in K'
\]
for every $k\ge 1$ and $0\le j\le n_k^H$.
\item[{\bf QT}] (Quadratic tangential return)
For each $k\ge 1$, writing $\zeta_{k+1}'=f^{n_k^H}(\zeta_k)$, we have
\[
Df^{n_k^T}(\zeta_{k+1}')v^u\in \mathbb R v^s.
\]
Moreover, there exists a constant $C_{T}>0$, independent of $k$, such that for every line segment
$
\gamma\subset \ell^u(\zeta_{k+1}')\cap K
$
containing $\zeta_{k+1}'$,
\[
\bigl|\pi^u_{k+1}(f^{n_k^T}(\gamma))\bigr|
\le
C_{T}
\bigl|\pi^s_{k+1}(f^{n_k^T}(\gamma))\bigr|^2.
\]
Here $v^s=(1,0)$ and $v^u=(0,1)$ are the horizontal and vertical unit vectors, respectively, $\ell^\sigma(z)$ denotes the lines in $K$ through $z$ parallel to $v^\sigma$ ($\sigma\in\{s,u\}$), and $\pi^\sigma_{k}$ denotes the orthogonal projection onto $\ell^\sigma(\zeta_{k})$.
\item[{\bf OE}] (Oscillatory and exponential growth of return times) 
There are real numbers $\alpha,\beta$ such that
\begin{align}\label{001}
\begin{split}
&1<\alpha <\beta<2,
\quad \dfrac{5}{2}\alpha-2-\frac{1}{\alpha\beta}<0,\quad
\lambda_s\lambda_u^{ \frac{9\beta-6}{2-\beta}}<1,\quad 
\lambda_u(\lambda_s\lambda_u^3)^{\alpha\beta-1} >1, \\
&\lambda_s ^{\frac{1}{2}}\lambda_u ^{\frac{6\beta}{2-\beta}-\frac{3}{2-\alpha}-\frac{9\beta-6}{2(2-\beta)}}<1, \quad
\lambda_s\lambda_u^{\frac{7\alpha}{2}}
>
\lambda_s\lambda_u^{\frac{9\beta-6}{2-\beta}},
\quad
\Bigl(\lambda_s\lambda_u^{\frac{7\alpha}{2}}\Bigr)^{\alpha\beta}
>
\lambda_s\lambda_u^{\frac{9\beta-6}{2-\beta}}.
\end{split}
\end{align}
Moreover, $n_k^H$ can be expressed with a positive integer $n_0^H$, according to whether $k$ is odd or even, as
\[
n_{2p}^H=\lfloor n_0^H \alpha^p\beta^p\rfloor  \quad \text{and} \quad n_{2p+1}^H=\lfloor n_0^H \alpha^{p+1}\beta^p\rfloor, 
\]
where $\lfloor x\rfloor$ is the greatest integer not exceeding a real number $x$.
\end{itemize}
\end{dfn}
\begin{remark}\label{rmk:0307}
The inequalities in \eqref{001} seem complicated at first sight. 
However, if one has the freedom of choosing $\alpha$ and $\beta$ sufficiently close to $1$, for instance, as in \cite{KLNS2022r}, then the inequalities reduce to ${\lambda}_s{\lambda}_u^3<1$, which is already required in {\bf H}.
\end{remark}

Recall that the \emph{Newhouse domain} is the open set of $\mbox{Diff}^r(M)$ such that every element of it admits a hyperbolic basic set displaying a $C^r$ robust homoclinic tangency and an area contracting periodic point. The existence of the $C^2$ Newhouse domain was shown by Newhouse for surface diffeomorphisms \cite{N1970, N1974}.   
The following result is established by refining the arguments of \cite[Theorem A]{KS2017}.
\begin{thm}\label{thm:KS}
There is a dense subset $\mathcal D$ of the Newhouse domain of $\mathrm{Diff}^r(M)$ such that every element of $\mathcal D$ satisfies the conditions {\bf DH}, {\bf H}, {\bf QT} and {\bf OE}.
\end{thm}

The proof of Theorem \ref{thm:KS} will be given throughout the rest of this section.
In Sections 3-5, we will prove the following theorem, which, together with Theorem \ref{thm:KS}, implies Theorem \ref{thm:main} immediately.
\begin{thm}\label{thm:main2}
Assume that $f$ satisfies {\bf DH}, {\bf H}, {\bf QT} and {\bf OE}. Then, there exist open sets $U\subset K$ and  $V\subset \mathbb R^2$ such that  
$( \frac{\log ||Df^n(x)v||}{n})_{n\ge 1}$ does not converge
for all $x\in U$ and nonzero $v\in V$.
\end{thm}

\subsection{Reduction to strongly dissipative Newhouse domain}
Let $p$ be a saddle periodic point of $f$ in $\mathrm{Diff}^r(M)$ and let 
$\lambda_s$ and $\lambda_u$ be 
the stable and unstable eigenvalues of  $Df^m(p)$ respectively, where $m$ is the period of $p$.
We say that $p$ is \emph{strongly dissipative} 
if $\lambda_s\lambda_u^3<1$.
Suppose that $\mathcal{N}^\ast$ is an open subset of $\mathrm{Diff}^r(M)$ 
consisting of elements $f$ which have a strongly dissipative saddle periodic point $p$ 
such that $W^{s}(p)$ and $W^{u}(p)$ have a persistent quadratic homoclinic tangency relative to a dense subset of $\mathcal{N}^\ast$.
We say that $\mathcal{N}^\ast$ is a \emph{strongly dissipative Newhouse domain}.

Let $\mathcal{N}$ be a Newhouse domain of $\mbox{Diff}^r(M)$. To prove Theorem \ref{thm:KS}, for every element $f$ of $\mathcal{N}$, we need to construct an arbitrarily small $C^r$ perturbation $g$ of $f$ such that $g$ satisfies the four conditions {\bf DH}, {\bf H}, {\bf QT} and {\bf OE}. Before that, let us make some reduction. More precisely, the following lemma suggests that by some renormalization process, we can always assume that the saddle satisfies the strongly dissipative condition.  

\begin{lem}
Let $\mathcal{N}$ be a Newhouse domain of $\mathrm{Diff}^r(M)$. Then, there is a dense subset $\mathcal{D}$ of $\mathcal{N}$ such that 
every diffeomorphism in $\mathcal{D}$ 
has homoclinic tangency and transverse intersection for 
a saddle periodic point that satisfies the strongly dissipative condition.
\end{lem}
\begin{proof}
Let $f$ be any element of $\mathcal{N}$, by an arbitrarily small $C^r$ perturbation (still denoted by $f$ for notational simplicity) if necessary, one can assume that $f$ has a homoclinic tangency, denoted by $q$, associated to some saddle periodic point, denoted by $p$. 
No generality is lost by assuming that $f$ is of $C^{\infty}$ class,  
the saddle $p$ is a dissipative (i.e., the product of the modulus of eigenvalues $Df(p)$ is smaller than 1) saddle fixed point
 and the tangency point $q$ is quadratic.
We now consider a one-parameter family $(f_{\mu})$ of $C^{\infty}$ diffeomorphisms with $f_{0}=f$ and such that the  tangency $q$ unfolds generically for $\mu=0$. 
By \cite[\S 3.4, Theorem 1]{PT1995}, we have a 
renormalization of return maps near $q$ in the following sense:
There is an integer $N_{0}>0$ such that, for any sufficiently large integer $n>0$, 
there exist a $C^{r}$ reparametrization  $\Theta_{n}:\mathbb{R}\to \mathbb{R}$ and a $\tilde \mu$-dependent $C^r$ coordinate change $\Phi_{n}:\mathbb{R}^2\to M$ satisfying the following conditions.
\begin{itemize}
\item
$\Theta_{n}'(\tilde \mu)>0$.
\item
For any $(\tilde\mu,\tilde x,\tilde y)\in \mathbb{R}\times \mathbb{R}^2$,
$(\Theta_n(\tilde \mu), \Phi_n(\tilde x,\tilde y))$ converges to $(0, q)$ as $n\rightarrow \infty$.
\item
For any $\tilde \mu\in \mathbb{R}$, the diffeomorphisms $\varphi_n$ 
on $\mathbb{R}^2$ defined by
\[
\varphi_{n}(\tilde{x},\tilde{y})= 
  \Phi_{n}^{-1}\circ f_{\mu_{n}}^{N_{0}+n}\circ \Phi_{n}(\tilde{x},\tilde{y}),\ \text{where}\ \mu_{n}:=\Theta_{n}(\tilde{\mu}),
\]
converge to the H\'enon-like map 
\[
\varphi_{\tilde{\mu}}(\tilde{x},\tilde{y})=( \tilde{y}, \tilde{y}^{2}+\tilde{\mu} ).
\]
as $n\rightarrow \infty$  in the $C^r$ topology.
\end{itemize}
Note that $\varphi_{-2}$ has a fixed point $P_{-2}=(2,2)$ 
and $D\varphi_{-2}(P_{-2})$ has the eigenvalues 
$\lambda_{u}=4$ and $\lambda_{s}=0$. 

We next consider the 2-parameter family of diffeomorphisms defined as
$$\varphi_{\tilde\mu,\tilde\nu}(\tilde{x},\tilde{y})=( \tilde{y}, \tilde{y}^{2}+\tilde{\nu} \tilde{x}+\tilde{\mu}), 
$$ 
which $C^{2}$-converges to $\varphi_{\tilde{\mu}}$ as $\tilde\nu\to 0$.\footnote{Note that $\varphi_{\mu,\nu}$ is 
topologically conjugate to the original H\'enon map $(x,y)\mapsto (1-ax^2+y, bx)$ via 
appropriate parameter and coordinate changes. } 
Thus, if $(\tilde\mu,\tilde\nu)$ is near $(-2,0)$,
$\varphi_{\tilde\mu,\tilde\nu}$ has the continuation of $P_{-2}$ given by
\[P_{\tilde\mu,\tilde\nu}=(y_{\tilde\mu,\tilde\nu}, y_{\tilde\mu,\tilde\nu}),\ \text{where}\ 
y_{\tilde\mu,\tilde\nu}=\frac{1-\tilde\nu+\sqrt{(1-\tilde\nu)^{2}-4\tilde\mu}}{2},\] 
and $D\varphi_{\tilde\mu,\tilde\nu}(P_{\tilde\mu,\tilde\nu})$ has the eigenvalues 
\[
\lambda_{s; \tilde\mu,\tilde\nu}=y_{\tilde\mu,\tilde\nu}-\sqrt{y_{\tilde\mu,\tilde\nu}^{2}+\tilde\nu},\ 
\lambda_{u; \tilde\mu,\tilde\nu}=y_{\tilde\mu,\tilde\nu}+\sqrt{y_{\tilde\mu,\tilde\nu}^{2}+\tilde\nu}.
\]
Since
the former is close to $0$ and the latter is close to $4$, as $(\tilde\mu,\tilde\nu)$ is contained in a small neighborhood of  $(-2,0)$, 
$P_{\tilde\mu,\tilde\nu}$ satisfies the strongly dissipative condition: 
\[\lambda_{s; \tilde\mu,\tilde\nu} \lambda^3_{u; \tilde\mu,\tilde\nu}<1.\]

Moreover, for any $\tilde\nu$ close to $0$, there exists
$\tilde\mu=\tilde\mu(\tilde\nu)$ near $-2$ such that
$W^{s}(P_{\tilde\mu,\tilde\nu})$ and $W^{u}(P_{\tilde\mu,\tilde\nu})$
have a transverse intersection and a quadratic tangency
$q_{\tilde\mu,\tilde\nu}$, and the latter unfolds generically in the
$\tilde\mu$-parameter family for fixed $\tilde\nu$.
See \cite[Proposition 3.2]{kls}.
Finally, all these properties are transferred to the original family
by the above renormalization.
\end{proof}

\subsection{Modification of diffeomorphisms for Theorem \ref{thm:KS}}
For any $f\in \mathcal{N}$, we construct a diffeomorphism arbitrarily $C^r$-close to $f$ and satisfying {\bf DH}, {\bf H}, {\bf QT} and {\bf OE} by 
modifying the diffeomorphism $f_{\underline{\boldsymbol{t}}}$ 
given in \cite[Section 7]{KS2017}.
In fact the $f_{\underline{\boldsymbol{t}}}$ already satisfies the former three conditions but does not satisfy {\bf OE}.
The condition {\bf OE} requires that the sequence $(n_k^H)_{k\geq 1}$ has exponential growth.
On the other hand,  
the corresponding sequence in \cite{KS2017} only has quadratic polynomial growth.
So we need to modify $f_{\underline{\boldsymbol{t}}}$ so as to satisfy {\bf OE}.

Let $a$ be a strongly dissipative saddle fixed point of $f$ with a persistent quadratic homoclinic tangency of 
period $m$.
We may assume that $a$ is a fixed point and all the eigenvalues of $Df(a)$ are positive if necessary by considering  
$f^{2m}$ instead of $f$.
Since $a$ is strongly dissipative, $Df(a)$ has the eigenvalues 
$\lambda$, $\sigma$ with $0<\lambda<1<\sigma$ and $\lambda\sigma^3<1$.
Moreover, one can suppose that $f$ satisfies the Sternberg open-dense condition \cite{st} concerning the eigenvalues by slightly modifying $f$ if needed.
Thus $f$ is $C^r$-linearizable in a neighborhood $K$ of $a$, that is, 
one can choose a $C^r$-coordinate on $K$ with respect to which $f$ is represented as $f(x,y)=(\lambda x,\sigma y)$ for any $(x,y)\in K$, where the origin $(0,0)$ corresponds to $a$. 
Then \[\lambda_{s}=\lambda\quad
\text{and}\quad \lambda_{u}=\sigma\]
satisfy the former part of {\bf H}.
Fix a nonempty closed set $K' \subset K$ including $a$.
Then, the latter part of {\bf H} comes from \cite[Section 7]{KS2017}: the block $1^{(z_k k^2)}$ in the critical-chain itineraries forces the orbit to stay for a long time in a fixed small linearizing neighborhood of the saddle fixed point $a$.

We next verify {\bf QT}. For each $k\ge1$, let
$
\widetilde{\boldsymbol{x}}_{k+1}'
=
f_{\underline{\boldsymbol{t}}}^{n_k^H}(\widetilde{\boldsymbol{x}}_k).
$
Then \cite[Lemma~8.2]{KS2017} gives
\[
Df_{\underline{\boldsymbol{t}}}^{n_k^T}(\widetilde{\boldsymbol{x}}_{k+1}')v^u\in \mathbb R v^s.
\]
Moreover, by \cite[Remark~7.7]{KS2017}, the curve $\widetilde J_k$ appearing in the proof of \cite[Lemma~8.2]{KS2017} belongs to a leaf of $\mathcal F_{\mathrm{loc}}(\Lambda)$, and hence its curvature is bounded by a constant independent of $k$.
Since $g_k(J_k)=f^{N_2}(\widetilde J_k)$ and $N_2$ is independent of $k$, \cite[Lemma~8.1(2)]{KS2017} yields, after shrinking $K$ if necessary, a constant $C_{T}>0$ independent of $k$ such that for every line segment
$
\gamma\subset \ell^u(\widetilde{\boldsymbol{x}}_{k+1}')\cap K
$
containing $\widetilde{\boldsymbol{x}}_{k+1}'$,
\[
\bigl|\pi^u_{k+1}(f_{\underline{\boldsymbol{t}}}^{n_k^T}(\gamma))\bigr|
\le
C_{T}
\bigl|\pi^s_{k+1}(f_{\underline{\boldsymbol{t}}}^{n_k^T}(\gamma))\bigr|^2,
\]
where $\pi^\sigma_{k+1}$ denotes the orthogonal projection onto the line through $\widetilde{\boldsymbol{x}}_{k+1}$ in the $v^\sigma$-direction.
Therefore $f_{\underline{\boldsymbol{t}}}$ satisfies {\bf QT}.

Fix $\alpha,\beta$
satisfying \eqref{001} (see Remark \ref{rmk:0307}).
Take 
a sufficiently large integer $n_0^H\geq 2$ 
such that 
$$\lambda\sigma^{\frac{9\beta-6+18(n_0^H)^{-1}}{2-\beta}}<1.
$$
Here we consider the two sequences defined by
\begin{equation}\label{eqn_n^H}
n_{2p}^H=\lfloor n_0^H\alpha^p\beta^p\rfloor\quad\text{and}\quad n_{2p+1}^H=\lfloor n_0^H\alpha^{p+1}\beta^p\rfloor,
\end{equation}
where $\lfloor x\rfloor$ is the greatest integer not exceeding a real number $x$.
One can take $n_0^H$ so that $n_k^H>1$ for any $k\in \mathbb{N}$.
Hence, 
we obtain {\bf OE}.

Let $f_{\underline{\boldsymbol{t}}}$ be a diffeomorphism arbitrarily $C^r$-close to $f$ and satisfying the conditions given in \cite{KS2017}.
For the positive integers $N_0$, $N_1$, $N_2$, $\widehat n_k=O(k)$, $\widehat i_{k+1}=O(k)$ 
defined in \cite{KS2017}, we set 
\begin{equation}\label{eqn_n^T}
n_k^T=N_0+N_1+N_2+\widehat n_k+\widehat i_{k+1}.
\end{equation}
Then one can choose a sequence $(\widetilde{\boldsymbol{x}}_k)_{k\geq 1}$ in $K$ corresponding to $(\zeta_k)_{k\geq 1}$ in {\bf DH} 
with $f_{\underline{\boldsymbol{t}}}^{n_k}(\widetilde{\boldsymbol{x}}_k)= \widetilde{\boldsymbol{x}}_{k+1}$, 
where $n_k=n_k^H+n_k^T$. 
The advantage of $f_{\underline{\boldsymbol{t}}}$ is the freedom in choosing the binary code associated 
with $\widetilde{\boldsymbol{x}}_k$.
In particular, $(\widetilde{\boldsymbol{x}}_k)_{k\geq 1}$ can be taken so that 
\begin{equation}\label{eq:251012}
f_{\underline{\boldsymbol{t}}}^j(\widetilde{\boldsymbol{x}}_k)\in K\quad \text{for every $j=0,\dots,n_k^H$}
\end{equation}
if necessary replacing the perturbation sequence $\underline{\boldsymbol{t}}$ by 
a certain sequence arbitrarily close to $\underline{\boldsymbol{t}}$ (see \cite[Equation (7.19)]{KS2017}).
By \eqref{eqn_n^H} and \eqref{eqn_n^T}, we have $\lim\limits_{k\to \infty}\dfrac{n_k^H}{n_k}=1$, 
which implies {\bf DH}.

Table \ref{tb_comparison} presents the correspondence between the notations used in 
{\bf DH}, {\bf H}, {\bf QT}, {\bf OE} and those in the modified $f_{\underline{\boldsymbol{t}}}$.
\setlength{\tabcolsep}{13pt}
\renewcommand{\arraystretch}{1.4}

\begin{table}[H]
\caption{Comparison of notations}
\begin{center}
\begin{tabular}{|c|c|}\hline
in {\bf DH}, {\bf H}, {\bf QT}, {\bf OE} & in the modified $f_{\underline{\boldsymbol{t}}}$\\ \hline
$\zeta_k$  &  $\widetilde{\boldsymbol{x}}_k$ \\ \hline
$n_k^H$  &  \eqref{eqn_n^H}\\ \hline
$n_k^T$  &  \eqref{eqn_n^T}\\ \hline
 $\lambda_{s}$  &  $\lambda$ \\ \hline
 $\lambda_{u}$ &  $\sigma$ \\ \hline
\end{tabular}
\label{tb_comparison}
\end{center}
\end{table}

\begin{remark}\label{rmk:1012}
In Sections \ref{Sec3}-\ref{s:5}, we construct the open set $U$ in Theorem \ref{thm:main2} by using the rectangles $U_{k,m}\subset K$ introduced in Section \ref{Sec3}, each centered at $\zeta_{k+m}=\widetilde{\boldsymbol{x}}_{k+m}$. 
Specifically, we will set $U=U_{k_0,0}$ with some positive integer $k_0$.
Moreover, we show that $f^{n_{k+m}}(U_{k,m})\subset U_{k,m+1}$; from the proof of this claim (in Section \ref{Sec3}), one readily verifies that $f^j(U_{k,m})$ is contained in the neighborhood $K$ of the saddle point $a$ for any $j=0,\ldots ,n_{k+m}^H$, according to \eqref{eq:251012}.
Therefore, due to {\textbf{DH}}, one can show that the Birkhoff time average of any continuous function $\varphi:M\to\mathbb R$
exists for every
$x\in U$:
\[
\lim_{n\to\infty}\frac{1}{n}\sum_{j=0}^{n-1}\varphi (f^j(x)) = \varphi (a).
\]
(In particular, the Dirac measure $\delta_a$ is the so-called Dirac physical measure of saddle type.)
See \cite[Section 4.4]{KLNS2022} for further details.
\end{remark}

\section{Modified rectangle lemma}\label{Sec3}
As explained in Section~\ref{Sec2}, the construction in \cite{KS2017} provides the geometric ingredients underlying {\bf DH}, {\bf H}, and {\bf QT} in the Newhouse domain. It also yields a sequence of mutually disjoint rectangles $(R_k)_{k\ge k_0}$ such that the center of $R_k$ is $\zeta_k$ and $f^{n_k}(R_k) \subset R_{k+1}$ (called Rectangle Lemma there).
A similar family of rectangles satisfying the same forward inclusion also appeared in \cite{KLNS2022} for a more concrete class of surface diffeomorphisms, and in that setting smaller-scale sets $(U_{k,m})_{m\ge0}$ were built inside those rectangles as an important step to show the non-existence of Lyapunov exponents.
In the present situation, however, the rectangles $R_k$ from \cite{KS2017} are not located in the neighborhood $K$ where the linearizing coordinate is available. Since this makes the calculations in Sections~\ref{s:5} and \ref{s:6} prohibitively difficult, we instead construct a family of sets $(U_{k,m})_{m\ge0}$ in a different region adapted to the linearizing coordinate.

Suppose that $f$ satisfies {\bf DH}, {\bf H}, {\bf QT} and {\bf OE} with the same notations in Section~\ref{Sec2}.
In this section, for every sufficiently large positive integer $k$ and every non-negative integer $m$, we aim to define a rectangle $U_{k,m}$ satisfying $f^{n_{k+m}}(U_{k,m})\subset U_{k,m+1}$.

\subsection{Construction of $U_{k,m}$}
According to the third inequality of \eqref{001} in {\bf OE}, we are allowed to fix $k_0\in\mathbb{N}$ large enough such that
\begin{equation}\label{002}
\lambda_s\lambda_u^{\frac{9\beta-6+18(n_{k_0}^H)^{-1}}{2-\beta}}<1.
\end{equation}
Note that \eqref{002} also holds for every subscript greater than $k_0$. For every integer $k\ge k_0$, let
\begin{equation}\label{003}
b_k=\lambda_u^{-3\sum_{i=0}^\infty \frac{n_{k+i}^H}{2^{i}}}
\quad\mbox{and}\quad
\epsilon_k=\Big(\lambda_s\lambda_u^{\frac{9\beta-6+18(n_k^H)^{-1}}{2-\beta}}\Big)^{n_k^H}.
\end{equation}
Thus, by \eqref{002} we have $\epsilon_k\to 0$ when $k\to +\infty$. Moreover, by direct calculation, we have
\begin{equation}\label{006}
\lambda_u^{3n_k^H}b_{k}=\lambda_u^{-3\sum_{i=1}^\infty \frac{n^H_{k+i}}{2^{i}}}=\sqrt{b_{k+1}}.
\end{equation}
Finally, for every non-negative integer $m$, let us define $\epsilon_{k,m}$ as 
\begin{equation}\label{004}
\epsilon_{k,m}=\epsilon_k^{(\alpha\beta)^{\lfloor m/2 \rfloor}}.
\end{equation}
Note that $\epsilon_{k,m} \neq \epsilon_{k+m}$ in general. 
These quantities $b_{k+m}$ and $\epsilon_{k,m}$ will be used to construct the desired rectangle $U_{k,m}$. Before doing so, we state the following claim, which will play a key role in the next subsection.

\begin{clm}\label{007} For every sufficiently large integer $k$ and every non-negative integer $m$, it holds that
\[\epsilon_{k,m}\le\lambda_u^{n_{k+m}^H}\epsilon_{k,m+1}.\]
\end{clm}
\begin{proof}
By the definition of $\epsilon_{k,m}$ in \eqref{004}, it is equivalent to show that
\[
\epsilon_k^{(\alpha\beta)^{\lfloor m/2\rfloor}}
\le
\lambda_u^{n_{k+m}^H}\epsilon_k^{(\alpha\beta)^{\lfloor (m+1)/2\rfloor}}.
\]
If $m$ is even, then $\lfloor m/2\rfloor=\lfloor (m+1)/2\rfloor$, and there is nothing to prove since $\lambda_u>1$.
Thus, we only need to consider the case where $m$ is odd, say $m=2t+1$ with $t\ge0$.
Then it is enough to show that
\begin{equation}\label{100}
\lambda_u^{n_{k+2t+1}^H}\epsilon_k^{(\alpha\beta)^t(\alpha\beta-1)}\ge 1
\end{equation}
for every sufficiently large $k$ and every $t\in\mathbb{N}$.

It holds that
\begin{equation}\label{eq:Sec3-Claim007-lower}
n_{k+2t+1}^H>\alpha(\alpha\beta)^t n_k^H-1
\end{equation}
for every $k\ge1$ and every $t\ge0$.
Indeed, if $k=2q$ is even, then by {\bf OE},
\[
n_{k+2t+1}^H=\lfloor n_0^H\alpha^{q+t+1}\beta^{q+t}\rfloor
>
\alpha^{t+1}\beta^t n_k^H-1
=
\alpha(\alpha\beta)^t n_k^H-1.
\]
If $k=2q+1$ is odd, then
\[
n_{k+2t+1}^H=\lfloor n_0^H\alpha^{q+t+1}\beta^{q+t+1}\rfloor
>
\beta(\alpha\beta)^t n_k^H-1
>
\alpha(\alpha\beta)^t n_k^H-1,
\]
because $\beta>\alpha$.
This proves \eqref{eq:Sec3-Claim007-lower}.

Using \eqref{eq:Sec3-Claim007-lower}, we obtain
\begin{align}\label{103}
\begin{split}
\lambda_u^{n_{k+2t+1}^H}\epsilon_k^{(\alpha\beta)^t(\alpha\beta-1)}
&>
\lambda_u^{-1}\lambda_u^{\alpha(\alpha\beta)^t n_k^H}
\Big(\lambda_s\lambda_u^{\frac{9\beta-6+18(n_k^H)^{-1}}{2-\beta}}\Big)^{n_k^H(\alpha\beta)^t(\alpha\beta-1)}\\
&=
\lambda_u^{-1}
\Big(
\lambda_u^\alpha
\Big(\lambda_s\lambda_u^{\frac{9\beta-6+18(n_k^H)^{-1}}{2-\beta}}\Big)^{\alpha\beta-1}
\Big)^{(\alpha\beta)^t n_k^H}\\
&\ge
\lambda_u^{-1}
\Big(\lambda_u(\lambda_s\lambda_u^3)^{\alpha\beta-1}\Big)^{(\alpha\beta)^t n_k^H},
\end{split}
\end{align}
where in the last inequality we used $\alpha>1$ and
\[
\frac{9\beta-6+18(n_k^H)^{-1}}{2-\beta}>3.
\]
By the fourth inequality of \eqref{001}, the base in the last line of \eqref{103} is greater than $1$.
Hence the right-hand side of \eqref{103} is greater than $1$ for every sufficiently large $k$, uniformly in $t$.
Thus \eqref{100} holds, and the proof is complete.
\end{proof}

Now, let us construct the rectangles $U_{k,m}$. Recall that by condition {\bf H}, we know that $\zeta_{k+m}$ is contained in an open set $K$ where a $C^r$ coordinate chart was defined.
Let $\zeta_{k+m}=(\zeta^s_{k+m}, \zeta^u_{k+m})$ and define
\begin{align*}
\ell^s_{k+m}&:=\Big[\zeta^s_{k+m}-\epsilon_{k,m}\sqrt{b_{k+m}}, \quad  \zeta^s_{k+m}+\epsilon_{k,m}\sqrt{b_{k+m}}\Big]\times\{\zeta^u_{k+m}\},
\\
\ell^u_{k+m}&:=\{\zeta^s_{k+m}\}\times\Big[\zeta^u_{k+m}-\frac{1}{4}\epsilon_{k,m}b_{k+m}, \quad  \zeta^u_{k+m}+\frac{1}{4}\epsilon_{k,m}b_{k+m}\Big].
\end{align*}
Obviously, $\ell^s_{k+m}$ and $\ell^u_{k+m}$ are horizontal and vertical line segments parallel to the coordinate axes. Now, we define the rectangle $U_{k,m}$ as
\begin{align*}
U_{k,m}:=&\Big[\zeta^s_{k+m}-\epsilon_{k,m}\sqrt{b_{k+m}}, \quad  \zeta^s_{k+m}+\epsilon_{k,m}\sqrt{b_{k+m}}\Big]\\
&\times \Big[\zeta^u_{k+m}-\frac{1}{4}\epsilon_{k,m}b_{k+m}, \quad  \zeta^u_{k+m}+\frac{1}{4}\epsilon_{k,m}b_{k+m}\Big], 
\end{align*}
See Figure \ref{fig} below. For simplicity, we call $2\epsilon_{k,m}\sqrt{b_{k+m}}$ the \emph{width} of $U_{k,m}$ and $\frac{1}{2}\epsilon_{k,m}b_{k+m}$ the \emph{height} of $U_{k,m}$.

In addition, for every $x=(x^s,x^u)\in \ell_{k+m}^u$, let 
\[\ell^s_{k+m}(x)=\Big[\zeta^s_{k+m}-\epsilon_{k,m}\sqrt{b_{k+m}}, \quad  \zeta^s_{k+m}+\epsilon_{k,m}\sqrt{b_{k+m}}\Big]\times\{x^u\}.\]
Thus, it follows immediately that
\[
U_{k,m}=\bigcup_{x\in \ell_{k+m}^u}\ell _{k+m}^s(x).
\]

\begin{lem}\label{eq:Sec3-Usubset}
For every sufficiently large integer $k$ and every non-negative integer $m$, it holds that
\begin{equation*}
U_{k,m}\subset \bigcap_{j=0}^{n_{k+m}^H}f^{-j}(K)
\end{equation*}
\end{lem}
\begin{proof}
Since $M$ is compact and $K$ is open, if $k$ is sufficiently large, then
\[
2\epsilon_k < \rho:=\mathrm{dist}(K',M\setminus K),
\]
where $K'$ is the non-empty closed subset of $K$ in {\bf H}.
Fix such a $k$. For every integer $m\ge 0$ and $x=(x^s,x^u)\in U_{k,m}$,
we prove by induction on $j=0,\dots,n_{k+m}^H$ that
\[
f^j(x)\in K,
\quad
\bigl|(f^j(x))^s-(f^j(\zeta_{k+m}))^s\bigr|
\le \epsilon_k,
\quad
\bigl|(f^j(x))^u-(f^j(\zeta_{k+m}))^u\bigr|
\le \frac14\epsilon_k.
\]

For $j=0$, since $x\in U_{k,m}$, we have
\[
|x^s-\zeta_{k+m}^s|
\le
\epsilon_{k,m}\sqrt{b_{k+m}}
\le
\epsilon_k
\]
and
\[
|x^u-\zeta_{k+m}^u|
\le
\frac14\epsilon_{k,m}b_{k+m}
\le
\frac14\epsilon_k.
\]
Hence
\[
\|x-\zeta_{k+m}\|<2\epsilon_k<\rho.
\]
Since $\zeta_{k+m}\in K'$, this implies $x\in K$.

Assume now that $f^i(x)\in K$ for every $0\le i\le j-1$, where $1\le j\le n_{k+m}^H$.
Since also $f^i(\zeta_{k+m})\in K$ for $0\le i\le j$, it follows from {\bf H} that
\[
\bigl|(f^j(x))^s-(f^j(\zeta_{k+m}))^s\bigr|
=
\lambda_s^j|x^s-\zeta_{k+m}^s|
\le
\epsilon_{k,m}\sqrt{b_{k+m}}
\le
\epsilon_k,
\]
and
\[
\bigl|(f^j(x))^u-(f^j(\zeta_{k+m}))^u\bigr|
=
\lambda_u^j|x^u-\zeta_{k+m}^u|
\le
\frac14\epsilon_{k,m}b_{k+m}\lambda_u^{n_{k+m}^H}.
\]
By \eqref{006},
\[
b_{k+m}\lambda_u^{n_{k+m}^H}
=
\lambda_u^{-2n_{k+m}^H}\sqrt{b_{k+m+1}}
<1,
\]
and therefore
\[
\bigl|(f^j(x))^u-(f^j(\zeta_{k+m}))^u\bigr|
<
\frac14\epsilon_{k,m}
\le
\frac14\epsilon_k.
\]
Hence
\[
\|f^j(x)-f^j(\zeta_{k+m})\|<2\epsilon_k<\rho.
\]
Since $f^j(\zeta_{k+m})\in K'$, we conclude that $f^j(x)\in K$.
This proves the claim.
\end{proof}

\subsection{Modified rectangle lemma}
Now, we are in the position to prove the following main result of Section ~\ref{Sec3}.

\begin{lem}\label{lem:mrl} Assume that $f\in\mathrm{Diff}^r(M)$ satisfies {\bf DH}, {\bf H}, {\bf QT} and {\bf OE}. Then, for every sufficiently large integer $k$ and every non-negative integer $m$, it holds that
\[f^{n_{k+m}}(U_{k,m})\subset U_{k,m+1}.\]
\end{lem}

The proof of this lemma follows immediately from the next two claims. Let $\frac{1}{2} U_{k,m}$ denote the rectangle defined in the same way as $U_{k,m}$ but replacing its width and height by half of those of $U_{k,m}$.

\begin{figure}[hbt]
\centering
\scalebox{0.65}{\includegraphics[clip]{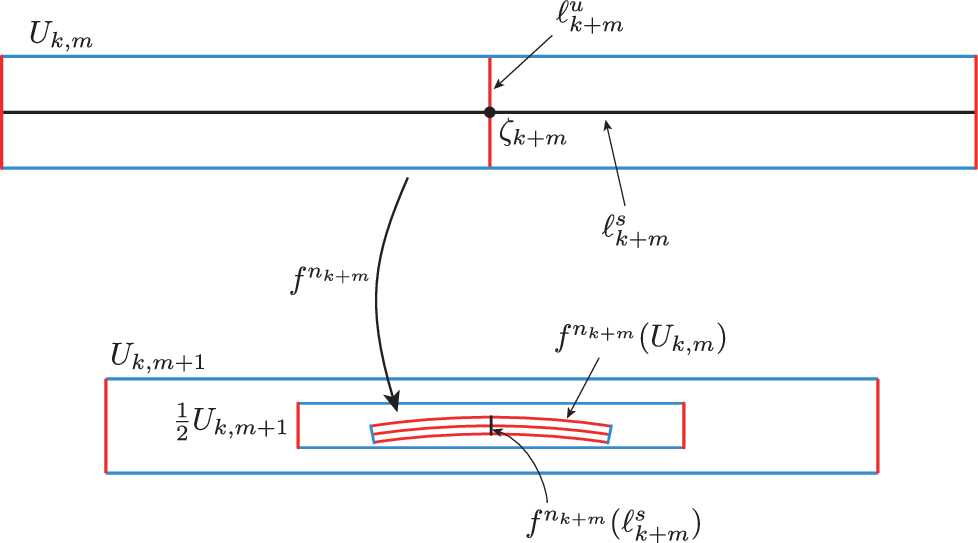}}
\caption{The modified rectangle lemma} 
\label{fig}
\end{figure}

\begin{clm}\label{005}For every non-negative integer $m$, we have
\[f^{n_{k+m}}(\ell_{k+m}^u)\subset \dfrac{1}{2} U_{k,m+1}.\]
\end{clm}
\begin{clm}\label{008} For every non-negative integer $m$, we have
\[\dfrac{|f^{n_{k+m}}(\ell^s_{k+m}(x))|}{\mathrm{diam}(U_{k,m+1})}<\dfrac{1}{8}\]
for every $x\in\ell^u_{k+m}$.
\end{clm}
\begin{proof}[Proof of Claim \ref{005}]
We use 
\[\pi^s_{k+m}:U_{k,m}\to \ell_{k+m}^s
\quad\mbox{and}\quad
\pi^u_{k+m}:U_{k,m}\to \ell_{k+m}^u
\]
to denote the orthogonal projections to $\ell_{k+m}^s$ and $\ell_{k+m}^u$ respectively. Notice that $f^{n_{k+m}}(\zeta_{k+m})=\zeta_{k+m+1}$ by condition {\bf DH}. Thus, to prove the claim, let us estimate the lengths of the orthogonal projections of $f^{n_{k+m}}(\ell_{k+m}^u)$. Define
\[\Lambda=\sup\big\{\|Df(x)v\|:x\in M,\ v\in T_xM,\ \|v\|=1\big\}.\]

For the orthogonal projections to $\ell^s_{k+m+1}$, since the segment $\ell_{k+m}^u$ stays inside $K$ for the first
$n_{k+m}^H$ iterates by Lemma \ref{eq:Sec3-Usubset}, we have 
\begin{align*}
|\pi^s_{k+m+1}(f^{n_{k+m}}(\ell^u_{k+m}))|
&\le |f^{n_{k+m}}(\ell^u_{k+m})| = |f^{n_{k+m}^T}\circ f^{n_{k+m}^H}(\ell^u_{k+m})|\\
&\le \Lambda^{n_{k+m}^T}\lambda_u^{n_{k+m}^H}|\ell_{k+m}^u|
\le \Lambda^{n_{k+m}^T}\lambda_u^{n_{k+m}^H}\cdot \dfrac{1}{2}\epsilon_{k,m}b_{k+m}\\
&= \dfrac{1}{2}\epsilon_{k,m}\Lambda^{n_{k+m}^T}\dfrac{\sqrt{b_{k+m+1}}}{\lambda_u^{2n_{k+m}^H}},
\end{align*}
where the last equality comes from \eqref{006}. By applying Claim \ref{007} in the last inequality, we obtain
\begin{align}\label{105}
\begin{split}
|\pi^s_{k+m+1}(f^{n_{k+m}}(\ell^u_{k+m}))|
&\le \dfrac{1}{2}\epsilon_{k,m+1}\sqrt{b_{k+m+1}}\dfrac{\Lambda^{n_{k+m}^T}}{\lambda_u^{n_{k+m}^H}}\\
&\le\dfrac{1}{2}\epsilon_{k,m+1}\sqrt{b_{k+m+1}}.
\end{split}
\end{align}
The last inequality holds for every sufficiently large $k$ because by {\bf DH}, we have
\[n^T_{k+m}\log \Lambda-n^H_{k+m}\log\lambda_u=n^H_{k+m}\left(\dfrac{n^T_{k+m}}{n^H_{k+m}}\log \Lambda-\log\lambda_u\right)\to -\infty,\ (k\to\infty)\]
which gives $\Lambda^{n_{k+m}^T}/\lambda_u^{n_{k+m}^H}<1$ for sufficiently large $k$.

For the orthogonal projections to $\ell^u_{k+m+1}$, by {\bf QT} we have
\begin{align*}
|\pi^u_{k+m+1}(f^{n_{k+m}}(\ell^u_{k+m}))|
&\le C_T |\pi^s_{k+m+1}(f^{n_{k+m}}(\ell^u_{k+m}))|^2
\\
&\le C_T  \Big(\dfrac{1}{2}\epsilon_{k,m+1}\sqrt{b_{k+m+1}}\Big)^2 
= \dfrac{C_T}{4} \epsilon^2_{k,m+1}b_{k+m+1}.
\end{align*}
But the height of $U_{k,m+1}$ is $\frac{1}{2}\epsilon_{k,m+1}b_{k+m+1}$ by our construction. Hence, the ratio of $|\pi^u_{k+m+1}(f^{n_{k+m}}(\ell^u_{k+m}))|$ to the height of $U_{k,m+1}$ is less than
$
\frac{1}{2}C_T\epsilon_{k,m+1}.
$
On the other hand, by the definition of
$\epsilon_{k,m+1}$, we have
\[
\epsilon_{k,m+1}\le \epsilon_k
\]
for every $m\ge 0$. Since $\epsilon_k\to 0$ as $k\to\infty$, the above ratio tends to $0$ uniformly in $m$. In particular, it is less than $\frac{1}{2}$ for every sufficiently large $k$ and every $m$. Combining \eqref{105}, we obtain
\[f^{n_{k+m}}(\ell^u_{k+m})\subset\frac{1}{2}U_{k,m+1}\] which completes the proof of Claim \ref{005}.
\end{proof}

\begin{proof}[Proof of Claim \ref{008}]
On the one hand, by Lemma \ref{eq:Sec3-Usubset}, for every $x\in\ell^u_{k+m}$, we have 
\begin{align*}
|f^{n_{k+m}}(\ell^s_{k+m}(x))|
&=
|f^{n^T_{k+m}}\circ f^{n^H_{k+m}}(\ell^s_{k+m}(x))|\\
&\le
\Lambda^{n_{k+m}^T}\lambda_s^{n_{k+m}^H}|\ell^s_{k+m}(x)|\\
&=
\Lambda^{n_{k+m}^T}\lambda_s^{n^H_{k+m}}\cdot 2\epsilon_{k,m}\sqrt{b_{k+m}}.
\end{align*}
On the other hand, according to our construction,
\[
\mathrm{diam}(U_{k,m+1})\ge \frac{1}{4}\epsilon_{k,m+1}b_{k+m+1}.
\]
Therefore, it is enough to prove that
\begin{equation}\label{009}
\sup_{m\ge 0}
\Lambda^{n_{k+m}^T}\lambda_s^{n^H_{k+m}}
\frac{\epsilon_{k,m}\sqrt{b_{k+m}}}{\epsilon_{k,m+1}b_{k+m+1}}
\longrightarrow 0
\qquad (k\to\infty).
\end{equation}
Set
\[
R_{k,m}:=
\Lambda^{n_{k+m}^T}\lambda_s^{n^H_{k+m}}
\frac{\epsilon_{k,m}\sqrt{b_{k+m}}}{\epsilon_{k,m+1}b_{k+m+1}}.
\]
By the fifth inequality of \eqref{001}, choose $\beta'>\beta$ so close to $\beta$ that
\begin{equation}\label{015}
\theta:=\lambda_s^{\frac12}
\lambda_u^{\frac{6\beta'}{2-\beta'}-\frac{3}{2-\alpha}-\frac{9\beta-6}{2(2-\beta)}}<1.
\end{equation}
For every $k,i\ge0$, condition {\bf OE} gives
\begin{equation}\label{351}
n_{k+2i}^H>(\alpha\beta)^i n_k^H-1
\qquad\text{and}\qquad
n_{k+2i+1}^H>\alpha(\alpha\beta)^i n_k^H-1,
\end{equation}
and also
\[
n_{k+2i}^H<(\alpha\beta)^i(n_k^H+1)
\qquad\text{and}\qquad
n_{k+2i+1}^H<\beta(\alpha\beta)^i(n_k^H+1).
\]
Hence
\begin{align*}
\sum_{i=0}^\infty \frac{n_{k+i}^H}{2^i}
&=
\sum_{j=0}^\infty \frac{n_{k+2j}^H}{4^j}
+\frac12\sum_{j=0}^\infty \frac{n_{k+2j+1}^H}{4^j}\\
&>
\sum_{j=0}^\infty \frac{(\alpha\beta)^j n_k^H-1}{4^j}
+\frac12\sum_{j=0}^\infty \frac{\alpha(\alpha\beta)^j n_k^H-1}{4^j}\\
&=
\frac{1+\alpha/2}{1-\alpha\beta/4}\,n_k^H-2,
\end{align*}
and
\begin{align*}
\sum_{i=0}^\infty \frac{n_{k+i}^H}{2^i}
&<
\sum_{j=0}^\infty \frac{(\alpha\beta)^j (n_k^H+1)}{4^j}
+\frac12\sum_{j=0}^\infty \frac{\beta(\alpha\beta)^j (n_k^H+1)}{4^j}\\
&=
\frac{1+\beta/2}{1-\alpha\beta/4}\,(n_k^H+1).
\end{align*}
Since
\[
\frac{1+\alpha/2}{1-\alpha\beta/4}>\frac{2}{2-\alpha}
\qquad\text{and}\qquad
\frac{1+\beta/2}{1-\alpha\beta/4}<\frac{2}{2-\beta}<\frac{2}{2-\beta'},
\]
it follows that, for every sufficiently large $k$,
\[
\frac{2}{2-\alpha} n_k^H
\le
\sum_{i=0}^\infty \frac{n_{k+i}^H}{2^i}
\le
\frac{2}{2-\beta'} n_k^H.
\]
Combining this with the definition of $b_k$ in \eqref{003}, we obtain
\begin{equation}\label{010}
\lambda_u^{-\frac{6}{2-\beta'}n_k^H}\le b_k \le \lambda_u^{-\frac{6}{2-\alpha}n_k^H}
\end{equation}
for every sufficiently large $k$.
Furthermore,
since
\[
\frac{n_{2p+1}^H}{n_{2p}^H}\to \alpha,
\quad
\frac{n_{2p+2}^H}{n_{2p+1}^H}\to \beta,
\quad
n_j^H\to\infty
\quad \frac{n_j^T}{n_j^H}\to 0 , 
\]
by {\bf DH} and {\bf OE}, 
we have
\begin{equation}\label{eq:0314c1}
n_j^H \le n_{j+1}^H\le \beta' n_j^H,
\quad
n_j^H\le n_j^H+1\le 2n_j^H,
\quad\Lambda^{n_j^T}\le \theta^{-\,n_j^H/(4\beta')}
\end{equation}
for sufficiently large $j$.

Next, we claim that
$
\alpha\beta<\frac32.
$
Indeed, if $\alpha\beta\ge \frac32$, then the second inequality of \eqref{001} gives
$
0>\frac52\alpha-2-\frac1{\alpha\beta}\ge \frac52\alpha-2-\frac23,
$
and hence
$
\alpha<\frac{16}{15}.
$
On the other hand, the inequality
$
\lambda_s\lambda_u^{\frac{7\alpha}{2}}
>
\lambda_s\lambda_u^{\frac{9\beta-6}{2-\beta}}
$
in \eqref{001} implies
$
\beta<\frac{2(7\alpha+6)}{18+7\alpha}.
$
Therefore,
\[
\alpha\beta
<
\alpha\cdot \frac{2(7\alpha+6)}{18+7\alpha}
\le
\frac{16}{15}\cdot \frac{2(7\cdot \frac{16}{15}+6)}{18+7\cdot \frac{16}{15}}
<
\frac32,
\]
which is a contradiction.
Consequently, we get
\begin{equation}\label{eq:0314c3}
\epsilon_{k,m+1}\ge \epsilon_{k,m}^{3/2}
\qquad\text{for every }m\ge0.
\end{equation}
Indeed, since $0<\epsilon_{k,m}<1$, if $m$ is even, then
$
\epsilon_{k,m+1}=\epsilon_{k,m}\ge \epsilon_{k,m}^{3/2},
$
while if $m$ is odd, then
$
\epsilon_{k,m+1}=\epsilon_{k,m}^{\alpha\beta}\ge \epsilon_{k,m}^{3/2}.
$

Now fix $k$ so large that
\eqref{010} and \eqref{eq:0314c1} hold for every integer $j\ge k$.
Then, we obtain
$
\sqrt{b_{k+m}}
\le
\lambda_u^{-\frac{3}{2-\alpha}n_{k+m}^H}
$ and
$\frac{1}{b_{k+m+1}}
\le
\lambda_u^{\frac{6}{2-\beta'}n_{k+m+1}^H}
$ from \eqref{010}.
Therefore, it follows from \eqref{eq:0314c1} and \eqref{eq:0314c3} that
\begin{align*}
R_{k,m}
&\le
\Lambda^{n^T_{k+m}}\lambda_s^{n^H_{k+m}}
\frac{\epsilon_{k,m}}{\epsilon_{k,m+1}}
\lambda_u^{\frac{6}{2-\beta'}n^H_{k+m+1}-\frac{3}{2-\alpha}n^H_{k+m}}.
\\
&\le
\Lambda^{n^T_{k+m}}\lambda_s^{n^H_{k+m}}
\frac{\epsilon_{k,m}}{\epsilon_{k,m+1}}
\lambda_u^{\frac{6\beta'}{2-\beta'}n^H_{k+m}-\frac{3}{2-\alpha}n^H_{k+m}}.
\\
&\le
\Lambda^{n^T_{k+m}}
\frac{\Bigl(\lambda_s\lambda_u^{\frac{6\beta'}{2-\beta'}-\frac{3}{2-\alpha}}\Bigr)^{n^H_{k+m}}}{\sqrt{\epsilon_{k,m}}}.
\end{align*}
Hence, since $n_{k+m}^H\ge (\alpha\beta)^{\lfloor m/2\rfloor}n_k^H-1$ by \eqref{351}, setting $C:=\Bigl(\lambda_s\lambda_u^{\frac{6\beta'}{2-\beta'}-\frac{3}{2-\alpha}}\Bigr)^{-1}$, we have
\begin{align}\label{eq:008-theta}
\begin{split}
 R_{k,m}
&\le
C 
\Lambda^{n^T_{k+m}}
\frac{\Bigl(\lambda_s\lambda_u^{\frac{6\beta'}{2-\beta'}-\frac{3}{2-\alpha}}\Bigr)^{(\alpha\beta)^{\lfloor m/2\rfloor}n_k^H}}{\sqrt{\epsilon_{k,m}}}\\
&=
C 
\Lambda^{n^T_{k+m}}
\Biggl(
\lambda_s^{\frac12}
\lambda_u^{\frac{6\beta'}{2-\beta'}-\frac{3}{2-\alpha}-\frac{9\beta-6+18(n_k^H)^{-1}}{2(2-\beta)}}
\Biggr)^{(\alpha\beta)^{\lfloor m/2\rfloor}n_k^H}\\
&\le
C 
\Lambda^{n^T_{k+m}}
\theta^{(\alpha\beta)^{\lfloor m/2\rfloor}n_k^H},
\end{split}
\end{align}
where we used $\lambda_u >1$ in the last inequality.

Now we compare $(\alpha\beta)^{\lfloor m/2\rfloor}n_k^H$ with $n_{k+m}^H$.
If $m=2t$, then by {\bf OE} and \eqref{eq:0314c1},
\[
n_{k+m}^H=n_{k+2t}^H<(\alpha\beta)^t(n_k^H+1)\le 2(\alpha\beta)^t n_k^H.
\]
If $m=2t+1$, then by {\bf OE} and \eqref{eq:0314c1},
\[
n_{k+m}^H=n_{k+2t+1}^H<\beta(\alpha\beta)^t(n_k^H+1)\le 2\beta'(\alpha\beta)^t n_k^H.
\]
Therefore, for every sufficiently large $k$ and every $m\ge0$,
\[
(\alpha\beta)^{\lfloor m/2\rfloor}n_k^H\ge \frac{1}{2\beta'}\,n_{k+m}^H.
\]
Hence, since $0<\theta<1$, \eqref{eq:008-theta} and \eqref{eq:0314c1} yields
\begin{align*}
R_{k,m}
&\le
C 
\Lambda^{n^T_{k+m}}
\theta^{\,n_{k+m}^H/(2\beta')}\\
&\le
C 
\theta^{\,n_{k+m}^H/(4\beta')}\\
&\le
C 
\theta^{\,n_k^H/(4\beta')}.
\end{align*}
Since $n_k^H\to\infty$, the right-hand side tends to $0$ as $k\to\infty$, uniformly in $m\ge0$.
This proves \eqref{009}, and hence the claim.
\end{proof}

\section{Jacobian matrices of the return maps}
\label{s:5}

We fix the (constant) stable and unstable directions on $K$ along the standard basis:
\[
    E^s = \mathrm{span}(v^s),\quad E^u = \mathrm{span}(v^u),\qquad v^s=\begin{pmatrix}1 \\ 0 \end{pmatrix},
    v^u=\begin{pmatrix}0 \\ 1 \end{pmatrix}.
\]
We study the representation matrix $A^{(m)}$ of $Df^{n_{k_0}+n_{k_0+1}+\cdots+n_{k_0+m-1}}$ on $U_{k_0,0}$ with respect to $\{v^s,v^u\}$ (the Jacobian matrix) with a fixed integer $k_0$.
This matrix is the product 
\[
A^{(m)}=A_{m-1}\cdots A_0,
\]
where $A_l$ represents $Df^{n_{k_0+l}}$.
In this section, we will show that each $A_l$ takes the form
\begin{equation*}
    A_l=\begin{bmatrix}
        b_l^{11}\lambda_s^{n_{k_0+l}^H} & b_l^{12}\lambda_u^{n_{k_0+l}^H}\\
        b_l^{21}\lambda_s^{n_{k_0+l}^H} & b_l^{22}\lambda_s^{n_{k_0+l}^H}
    \end{bmatrix},
\end{equation*}
where the coefficients $b_l^{ij}$ are bounded independently of $n_{k_0+l}^H$, 
with bounds controlled by $n_{k_0+l}^T$.

\subsection{Constants $k_0$ and $\xi$}\label{ss:k0}

Introduce
\[
    \Lambda_0 :=     B,
    \qquad \Lambda_1 := \max\Big\{2,\frac{L}{B(B-1)}\Big\}
    \Lambda_0^2,
\]
 where $B:=\sup_{\sigma=\pm1,\,x\in M}\|Df^{\sigma}(x)\|$ and $L:=\mathrm{Lip}(Df)$.
 Set
\begin{align*}
&C_{f,m}:=\Lambda_1^{n_{k_0}^T+\cdots+n_{k_0+m-1}^T},\\
&\Lambda_{u}^{(m)}:=
\begin{cases}
\lambda_{u}^{n_{k_0+m-1}^H}\lambda_{s}^{n_{k_0+m-2}^H}\cdots \lambda_{u}^{n_{k_0}^H} \quad &(\text{$m$: odd})\\
\lambda_{u}^{n_{k_0+m-1}^H}\lambda_{s}^{n_{k_0+m-2}^H}\cdots \lambda_{s}^{n_{k_0}^H}\quad &(\text{$m$: even})
\end{cases},\\
&\Lambda_{s}^{(m)}:=
\begin{cases}
\lambda_{s}^{n_{k_0+m-1}^H}\lambda_{u}^{n_{k_0+m-2}^H}\cdots \lambda_{s}^{n_{k_0}^H} \quad &(\text{$m$: odd})\\
\lambda_{s}^{n_{k_0+m-1}^H}\lambda_{u}^{n_{k_0+m-2}^H}\cdots \lambda_{u}^{n_{k_0}^H}\quad &(\text{$m$: even})
\end{cases}
\end{align*}
for $m\ge 1$.
Fix $\beta' \in(\beta,\frac{5}{4})$ (so that $\frac{3}{2-\beta'}-4<0$).
We then fix an even integer $k_0$ and $\xi\in(0,1)$ satisfying the following.
(The particular numerical constants appearing below are not essential. They are chosen only for convenience in the later estimates.)
\begin{itemize}
    \item The conclusion of Lemma~\ref{lem:mrl} holds.
\item For all integers $k\ge k_0$ and $m\ge 0$,
\begin{equation}\label{eq:1004}
\begin{split}
&(\alpha\beta)^{\lceil\frac{m}{2}\rceil-\lfloor\frac{m}{2}\rfloor}
\log\Bigl(\lambda_s^{1+(n_k^H)^{-1}}\lambda_u^{\frac{7\alpha}{2}}\Bigr)\\
&\qquad\ge
\log\Bigl(\lambda_s\lambda_u^{\frac{9\beta-6+18(n_k^H)^{-1}}{2-\beta}}\Bigr)
+(n_k^H)^{-1}(\alpha\beta)^{-\lfloor\frac{m}{2}\rfloor}
(\log(2\lambda_u)-\log\xi),
\end{split}
\end{equation}
    where $\lceil x\rceil$ is the least integer greater than or equal to $x$ and $\lfloor x\rfloor$ is the greatest integer less than or equal to $x$.
    \item It holds that
    \begin{equation}\label{eq:20251018b}
       \max\{\xi\Lambda_1^{2n_{k_0}^T},
      (\lambda_s\lambda_u^{-1})^{n_{k_0}^H}\Lambda_1^{2n_{k_0}^T}
      \}\le 10^{-4}.
    \end{equation}
\item For all integers $k\ge k_0$, the following hold:
\begin{subequations}
\begin{align}
&n_{k+1}^H \ge n_k^H,
\label{eq:20260314mono}
\\
&\Big(\frac{\lambda_s}{\lambda_u}\Big)^{n_k^H}
\le
\frac{\Lambda_1^{-2n_k^T}}{2},
\label{eq:20251101}
\\
&4\lambda_u^{\left(\frac{3}{2-\beta'}-4\right)n_{k}^H}
\le 1.
\label{eq:20260308}
\end{align}
\end{subequations}
\item For all integers $m\ge 0$
\begin{equation}\label{eq:lem53step3}
\tfrac{7}{2}\alpha-2
-(\alpha\beta)^{\lfloor m/2\rfloor-\lceil m/2\rceil}
+3(n_{k_0}^H)^{-1}(\alpha\beta)^{-\lceil m/2\rceil}
<\alpha.
\end{equation}
    \item 
    For all integers $m\ge 1$,
    \begin{equation}\label{eq:k0decay}
        \frac{\Lambda_{s}^{(m)}}{\Lambda_{u}^{(m)}}
        \le \frac{1}{12\cdot 3^{m} C_{f,m}^{2}\Lambda_1^{2n_{k_0+m}^T}}.
    \end{equation}
\end{itemize}

Such $k_0$ and $\xi$ exist.
Regarding \eqref{eq:1004} and \eqref{eq:20251018b}, by \eqref{001} we have
\[
\log\Bigl(\lambda_s\lambda_u^{\frac{7\alpha}{2}}\Bigr)
-\log\Bigl(\lambda_s\lambda_u^{\frac{9\beta-6}{2-\beta}}\Bigr)>0
\]
and
\[
\alpha\beta\log\Bigl(\lambda_s\lambda_u^{\frac{7\alpha}{2}}\Bigr)
-\log\Bigl(\lambda_s\lambda_u^{\frac{9\beta-6}{2-\beta}}\Bigr)>0.
\]
Hence one can choose positive constants $c_0,c_1$ such that
\[
0<c_0<c_1<
\min\Biggl\{
\log\Bigl(\lambda_s\lambda_u^{\frac{7\alpha}{2}}\Bigr)
-\log\Bigl(\lambda_s\lambda_u^{\frac{9\beta-6}{2-\beta}}\Bigr),\,
\alpha\beta\log\Bigl(\lambda_s\lambda_u^{\frac{7\alpha}{2}}\Bigr)
-\log\Bigl(\lambda_s\lambda_u^{\frac{9\beta-6}{2-\beta}}\Bigr)
\Biggr\}.
\]
Since $n_k^H\to\infty$ by {\bf OE} and $n_k^T/n_k^H\to0$ by {\bf DH}, after taking $k_0$ sufficiently large and even we have
\begin{align*}
&n_{k_0}^H\Biggl\{
(\alpha\beta)^{\lceil\frac{m}{2}\rceil-\lfloor\frac{m}{2}\rfloor}
\log\Bigl(\lambda_s^{1+(n_{k_0}^H)^{-1}}\lambda_u^{\frac{7\alpha}{2}}\Bigr)
-\log\Bigl(\lambda_s\lambda_u^{\frac{9\beta-6+18(n_{k_0}^H)^{-1}}{2-\beta}}\Bigr)
\Biggr\}
-\log(2\lambda_u)\\
&\qquad > n_{k_0}^Hc_1-\log(2\lambda_u)\\
&\qquad > n_{k_0}^Hc_0+\log(10^4)\\
&\qquad \ge 2n_{k_0}^T\log\Lambda_1+\log(10^4)
\end{align*}
for any integer $m\ge 0$ with a sufficiently large integer $k_0$.
We choose $\xi \in (0,1)$ so that $-\log\xi$ lies between the right-hand side and the left-hand side of the above inequality.
Since $n_k^H$ is increasing in $k$, the left-hand side of \eqref{eq:1004} increases and the correction term on the right-hand side decreases as $k$ increases. Hence, once \eqref{eq:1004} holds for $k=k_0$, it also holds for every $k\ge k_0$.
The second term in \eqref{eq:20251018b} is also bounded by $10^{-4}$ because of {\bf DH}.

\eqref{eq:20260314mono} holds because 
$
\frac{n_{2p+1}^H}{n_{2p}^H}\to\alpha>1$
and
$\frac{n_{2p+2}^H}{n_{2p+1}^H}\to\beta>1
$
from {\bf OE}.
\eqref{eq:20251101} follows from $\frac{n_k^T}{n_k^H}\to 0$ in {\bf DH}, and \eqref{eq:20260308} follows from $n_k^H\to \infty$.
\eqref{eq:lem53step3} holds because $\frac{7}{2}\alpha-2-1<\frac{7}{2}\alpha-2-\frac{1}{\alpha\beta}<\alpha$ by \eqref{001}.

Finally, we show that \eqref{eq:k0decay} holds for a sufficiently large even $k_0$.
Since
$
\frac{n_{2p+1}^H}{n_{2p}^H}\to \alpha$ and
$\frac{n_{2p+2}^H}{n_{2p+1}^H}\to \beta$ by {\bf OE},
 and since $\beta>\alpha>1$, after taking $k_0$ large enough we may assume that
\[
n_{k+1}^H\ge \frac{\alpha+1}{2}\,n_k^H
\qquad\text{for every }k\ge k_0.
\]
Hence
\[
n_{k+1}^H-n_k^H
\ge
\frac{\alpha-1}{\alpha+3}\,(n_{k+1}^H+n_k^H)
\qquad\text{for every }k\ge k_0.
\]
Therefore,
\begin{equation}\label{eq:0314b1}
\log\frac{\Lambda_u^{(m)}}{\Lambda_s^{(m)}}
=
\log\frac{\lambda_u}{\lambda_s}
\sum_{j=0}^{m-1}(-1)^{m-1-j}n_{k_0+j}^H
\ge
\frac{\alpha-1}{\alpha+3}\log\frac{\lambda_u}{\lambda_s}
\sum_{j=0}^{m-1}n_{k_0+j}^H
\end{equation}
for every $m\ge1$.
On the other hand, by {\bf DH}, taking $k_0$ larger if necessary, we may assume that for every $k\ge k_0$,
\begin{equation}\label{eq:0314}
2(\log\Lambda_1)n_k^T
\le
\frac{\alpha-1}{8\beta'(\alpha+3)}\Bigl(\log\frac{\lambda_u}{\lambda_s}\Bigr)n_k^H.
\end{equation}
Summing \eqref{eq:0314} from $k=k_0$ to $k_0+m-1$, we obtain
\begin{equation}\label{eq:0314b2}
2\log C_{f,m}
=
2(\log\Lambda_1)\sum_{j=0}^{m-1}n_{k_0+j}^T
\le
\frac{\alpha-1}{8\beta'(\alpha+3)}\log\frac{\lambda_u}{\lambda_s}
\sum_{j=0}^{m-1}n_{k_0+j}^H .
\end{equation}
Moreover,
after enlarging $k_0$ if necessary, we may assume that
\[
n_{k+1}^H\le \beta' n_k^H
\qquad\text{for every }k\ge k_0,
\]
so that
\[
n_{k_0+m}^H\le \beta' n_{k_0+m-1}^H
\le
\beta'\sum_{j=0}^{m-1}n_{k_0+j}^H.
\]
Therefore \eqref{eq:0314} gives that
\begin{equation}\label{eq:0314b3}
2(\log\Lambda_1)n_{k_0+m}^T
\le
\frac{\alpha-1}{8(\alpha+3)}\log\frac{\lambda_u}{\lambda_s}
\sum_{j=0}^{m-1}n_{k_0+j}^H .
\end{equation}
Also, since $n_{k_0+j}^H\ge n_{k_0}^H$ for all $j\ge0$, taking $k_0$ larger again if necessary, we get
\begin{equation}\label{eq:0314b4}
m\log 3+\log 12
\le
m\log 36
\le
m\cdot n_{k_0}^H\frac{\alpha-1}{4(\alpha+3)}\log\frac{\lambda_u}{\lambda_s}
\le
\frac{\alpha-1}{4(\alpha+3)}\log\frac{\lambda_u}{\lambda_s}
\sum_{j=0}^{m-1}n_{k_0+j}^H .
\end{equation}
Combining \eqref{eq:0314b1}, \eqref{eq:0314b2}, \eqref{eq:0314b3} and \eqref{eq:0314b4}, for every $m\ge1$,
\[
2\log C_{f,m}+2(\log\Lambda_1)n_{k_0+m}^T+m\log 3+\log 12
\le
\frac{\alpha-1}{2(\alpha+3)}\log\frac{\lambda_u}{\lambda_s}
\sum_{j=0}^{m-1}n_{k_0+j}^H
\le
\log\frac{\Lambda_u^{(m)}}{\Lambda_s^{(m)}}.
\]
Exponentiating both sides, we obtain \eqref{eq:k0decay}.

We will show that the interior of $U_{k_0,0}$ is the nonempty open set $U$
in Theorem~\ref{thm:main2} on which Lyapunov exponents fail to exist.
For readability, we suppress $k_0$ from the notation throughout the rest
of the paper.

\subsection{Notation}

Let $U:=\mathrm{int}(U_{k_0,0})$.
Fix $x\in U$ and set
\[
    x_m = F^{(m)}(x),\quad
    x_m' = F_{m-1}'\circ F^{(m-1)}(x)
\]
for $m\ge 1$,
where
\[
    F^{(m)} = F_{m-1}\circ\cdots\circ F_0,\quad
    F_l = f^{n_{k_0+l}},\quad
    F_l' = f^{n_{k_0+l}^H}.
\]
For $m\ge 1$ we also set
\[
    N_m = n_{k_0}+\cdots+n_{k_0+m-1},\quad
    N_m' = N_{m-1}+n_{k_0+m-1}^H,
\]
with $N_0=N_0'=0$.  Then
\[
    Df^{N_m}(x)=DF_{m-1}(x_{m-1})\cdots DF_0(x_0),
    \quad
    DF_l(x_l)=Df^{n_{k_0+l}^T}(x_{l+1}')DF_l'(x_l),
\]
We finally set 
\[
X_m=\zeta_{k_0+m}=F^{(m)}(\zeta_{k_0}),\quad
X_m'=F_{m-1}'\circ F^{(m-1)}(\zeta_{k_0}).
\]
By Lemma~\ref{lem:mrl}, both $x_m$ and $X_m$ lie in $U_{k_0,m}$.


Let $A_l$ be the representation matrix of $DF_l(x_l)$ with respect to $\{v^s,v^u\}$ (namely, the Jacobian matrix of $F_l$ at $x_l$ in the local coordinates on $K$), so that
$A^{(m)}:=A_{m-1}\cdots A_0$ represents $DF^{(m)}(x)$.
We factor 
\[
A_l=T_l H_l,
\]
where $H_l$ represents $DF_l'(x_l)$
and $T_l$ represents $Df^{n_{k_0+l}^T}(x_{l+1}')$ with respect to $\{v^s,v^u\}$.
Notice that all of $x_l, x'_{l+1}=f^{n_{k_0+l}^H}(x_l), x_{l+1}=f^{n_{k_0+l}^T}(x_{l+1}')$ are in $K$ by {\bf H} and Lemma~\ref{lem:mrl}, so that these representation matrices are well-defined.
By Lemma \ref{eq:Sec3-Usubset},
\begin{equation}\label{eq:1104ex}
    H_l=\begin{bmatrix}
        \lambda_s^{n_{k_0+l}^H} & 0\\
        0 & \lambda_u^{n_{k_0+l}^H}
    \end{bmatrix}.
\end{equation}

\subsection{Lower order terms in $A_l$}

Denote by $\hat T_l$ the matrix $Df^{n_{k_0+l}^T}(X_{l+1}')$.
By {\bf QT}, $Df^{n_{k_0+l}^T}(X_{l+1}')E^u= E^s$,
so the $(2,2)$-entry of $\hat T_l$ vanishes.
This gives real numbers
$a_l^{11}$, $a_l^{12}$, $a_l^{21}$ such that
\begin{equation}\label{eq:1012}
    \max_{ij}|a_l^{ij}|\le\tfrac{1}{2}\Lambda_1^{n_{k_0+l}^T},\quad
    \min\{|a_l^{12}|,|a_l^{21}|\}\ge 2\Lambda_1^{-n_{k_0+l}^T},\quad
    \hat T_l=\begin{bmatrix}a_l^{11}&a_l^{12}\\a_l^{21}&0\end{bmatrix}.
\end{equation}
Indeed, since $\widehat T_l$ is invertible,
we have $a_l^{12}a_l^{21}\neq 0$ and
\[
\widehat T_l^{-1}v^s=\binom{0}{1/a_l^{12}},
\qquad
\widehat T_l^{-1}v^u=\binom{1/a_l^{21}}{-\,a_l^{11}/(a_l^{12}a_l^{21})},
\]
and hence
\[
\|\widehat T_l^{-1}\|
\ge\max\bigl\{\|\widehat T_l^{-1}v^s\|,\ \|\widehat T_l^{-1}v^u\|\bigr\}
\ge\max\bigl\{|a_l^{12}|^{-1},\ |a_l^{21}|^{-1}\bigr\}.
\]
On the other hand,
\[
\|\widehat T_l^{-1}\|
=\bigl\|Df^{-n_{k_0+l}^T}\bigl(f^{n_{k_0+l}^T}(X_{l+1}')\bigr)\bigr\|
\le \Big(\sup_{\sigma=\pm1,\,x\in M}\|Df^\sigma(x)\|\Big)^{n_{k_0+l}^T}
\le \Lambda_0^{\,n_{k_0+l}^T}.
\]
Since $\Lambda_1\ge 2\Lambda_0^2\ge 2\Lambda_0$, we have $B^{-n_{k_0+l}^T}\ge 2\Lambda_1^{-n_{k_0+l}^T}$, and hence the lower bounds in \eqref{eq:1012} follow.
For the upper bounds we simply note that
$|a_l^{ij}| \le \|\widehat T_l\|\le \Lambda_0^{n_{k_0+l}^T}$.

Note that
\[
\operatorname{Lip}(Df^{n_{k_0+l}^T})\le \frac{L}{B(B-1)} B^{2n_{k_0+l}^T}.
\]
In fact, it holds that $\|Df^n\|_\infty\le B^{n}$ and $\operatorname{Lip}(Df^{n})\le\|Df\|_\infty\operatorname{Lip}(Df^{n-1})+\operatorname{Lip}(Df)\|Df^{n-1}\|_\infty^2$,
so
one can see $\operatorname{Lip}(Df^n)\le \frac{L}{B(B-1)} B^{2n}$ for every $n\ge 1$ by induction.
Hence,
\[
\|T_l-\hat T_l\|
=
\|Df^{n_{k_0+l}^T}(x'_{l+1})-Df^{n_{k_0+l}^T}(X'_{l+1})\|
\le
\frac{L}{B(B-1)} \Lambda_0^{2n_{k_0+l}^T}\|\tilde x'_{l+1}\|,
\]
where
\[
\tilde x'_{l+1}:=X'_{l+1}-x'_{l+1}.
\]
In particular, by the definition of $\Lambda_1$, for each $i,j\in\{1,2\}$,
\[
|(T_l)_{ij}-(\hat T_l)_{ij}|
\le
\frac{L}{B(B-1)} \Lambda_0^{2n_{k_0+l}^T}\|\tilde x'_{l+1}\| \le \Lambda_1^{n_{k_0+l}^T}\|\tilde x'_{l+1}\|
\]

Denoting $O_l(\|\cdot\|)$ for terms bounded by
$\Lambda_1^{n_{k_0+l}^T}\|\cdot\|$,
\begin{equation}\label{eqlem:4.1}
    T_l=\begin{bmatrix}
        a_l^{11}+O_l(\|\tilde x_{l+1}'\|) & a_l^{12}+O_l(\|\tilde x_{l+1}'\|)\\
        a_l^{21}+O_l(\|\tilde x_{l+1}'\|) & O_l(\|\tilde x_{l+1}'\|)
    \end{bmatrix}.
\end{equation}
Combining \eqref{eq:1104ex} and \eqref{eqlem:4.1}, we get
\begin{equation}\label{eq:20250927a}
    A_l=\begin{bmatrix}
        (a_l^{11}+O_l(\|\tilde x_{l+1}'\|))\lambda_s^{n_{k_0+l}^H}
          & (a_l^{12}+O_l(\|\tilde x_{l+1}'\|))\lambda_u^{n_{k_0+l}^H}\\
        (a_l^{21}+O_l(\|\tilde x_{l+1}'\|))\lambda_s^{n_{k_0+l}^H}
          & O_l(\|\tilde x_{l+1}'\|)\lambda_u^{n_{k_0+l}^H}
    \end{bmatrix}.
\end{equation}

\subsection{Estimate of the $(2,2)$-entry of $A_l$}

To obtain the oscillation of finite time Lyapunov exponents, it is important to see that the $(2,2)$-entry $ O_l(||\tilde x_{l+1}'||) \lambda_{u}^{n_{k_0+l}^H}$ of $A_l$ is of order $\lambda_{s}^{n_{k_0+l}^H}$.
The following is the key lemma for that purpose.

\begin{lem}\label{lem:5.3}
For every integer $l\ge 0$, we have
\[
    \|\tilde x_{l+1}'\|\lambda_u^{n_{k_0+l}^H}\le\xi\lambda_s^{n_{k_0+l}^H}.
\]
\end{lem}
\begin{proof}
Write
\[
x_l=(x_l^s,x_l^u),\qquad X_l=(X_l^s,X_l^u)
\]
in the coordinate of $K$.
Then
\[
\tilde x_{l+1}'
=f^{n_{k_0+l}^H}(X_l) - f^{n_{k_0+l}^H}(x_l)
=\big(\lambda_{s}^{n_{k_0+l}^H}(X_l^s-x_l^s),\ \lambda_{u}^{n_{k_0+l}^H}(X_l^u-x_l^u)\big).
\]
Since both $x_l$ and $X_l$ are in the rectangle $U_{k_0,l}$ with the width $2\epsilon_{k_0,l}\sqrt{b_{k_0+l}}$ and the height $\frac12\epsilon_{k_0,l}b_{k_0+l}$, we get
\[
|X_l^s-x_l^s|\le 2\epsilon_{k_0,l}\sqrt{b_{k_0+l}},
\qquad
|X_l^u-x_l^u|\le \frac12\epsilon_{k_0,l}b_{k_0+l}.
\]
Therefore,
\begin{align*}
\|\tilde x_{l+1}'\|\lambda_{u}^{n_{k_0+l}^H}
&\le \lambda_{u}^{n_{k_0+l}^H}\Big(\lambda_{s}^{n_{k_0+l}^H}|X_l^s-x_l^s|+\lambda_{u}^{n_{k_0+l}^H}|X_l^u-x_l^u|\Big)\\
&\le 2\epsilon_{k_0,l}\sqrt{b_{k_0+l}}\lambda_{s}^{n_{k_0+l}^H}\lambda_{u}^{n_{k_0+l}^H}+\frac12\epsilon_{k_0,l}b_{k_0+l}\lambda_{u}^{2n_{k_0+l}^H}.
\end{align*}
To compare these two terms, note that $\lambda_s\lambda_u^3<1$ implies
$\lambda_{s}^{n_{k_0+l}^H}/\lambda_{u}^{n_{k_0+l}^H}
\le \lambda_u^{-4n_{k_0+l}^H}$.
On the other hand, \eqref{010} gives a uniform lower bound
$\sqrt{b_{k_0+l}}\ge \lambda_u^{-\frac{3}{2-\beta'}n_{k_0+l}^H}$. 
Thus, it follows from \eqref{eq:20260308} that
\[
\frac{2\sqrt{b_{k_0+l}}\lambda_{s}^{n_{k_0+l}^H}\lambda_{u}^{n_{k_0+l}^H}}{\frac12 b_{k_0+l}\lambda_{u}^{2n_{k_0+l}^H}}
= \frac{4}{\sqrt{b_{k_0+l}}}\frac{\lambda_{s}^{n_{k_0+l}^H}}{\lambda_{u}^{n_{k_0+l}^H}}
\le 4\lambda_u^{\left(\frac{3}{2-\beta'}-4\right)n_{k_0+l}^H}\le 1.
\]
Consequently, 
 we have
\[
\|\tilde x_{l+1}'\|\lambda_{u}^{n_{k_0+l}^H}\le \epsilon_{k_0,l}b_{k_0+l}\lambda_{u}^{2n_{k_0+l}^H}.
\]

It remains to show that
\begin{equation}\label{eq:lem53goal}
\epsilon_{k_0,l}b_{k_0+l}\lambda_{u}^{2n_{k_0+l}^H} \le \xi\lambda_{s}^{n_{k_0+l}^H}.
\end{equation}
The proof of this inequality is essentially a repetition of the argument in \cite{KLNS2022r}.
However, since the notation is slightly different, and for the convenience of the reader, we give a complete proof here.
The remaining argument has two parts: first, we use \eqref{eq:1004} to extract the gain coming from $\epsilon_{k_0,l}$ (see \eqref{eq:lem53upper}); second, we use {\bf OE} together with the definition of $b_k$ to convert this into the factor $\lambda_s^{n_{k_0+l}^H}$ (see \eqref{eq:lem53lower}).

By \eqref{eq:1004} with $k=k_0$ and $m=l$,
\[
(\alpha\beta)^{\lceil\frac{l}{2}\rceil-\lfloor\frac{l}{2}\rfloor}
\log\Bigl(\lambda_s^{1+(n_{k_0}^H)^{-1}}\lambda_u^{\frac{7\alpha}{2}}\Bigr)
\ge
\log\Bigl(\lambda_s\lambda_u^{\frac{9\beta-6+18(n_{k_0}^H)^{-1}}{2-\beta}}\Bigr)
+(n_{k_0}^H)^{-1}(\alpha\beta)^{-\lfloor\frac{l}{2}\rfloor}
(\log(2\lambda_u)-\log\xi).
\]
Multiplying this inequality by $(\alpha\beta)^{\lfloor l/2\rfloor}$, we get
\[
(\alpha\beta)^{\lceil\frac{l}{2}\rceil}
\log\Bigl(\lambda_s^{1+(n_{k_0}^H)^{-1}}\lambda_u^{\frac{7\alpha}{2}}\Bigr)
+(n_{k_0}^H)^{-1}\log\xi
\ge
(\alpha\beta)^{\lfloor\frac{l}{2}\rfloor}
\log\Bigl(\lambda_s\lambda_u^{\frac{9\beta-6+18(n_{k_0}^H)^{-1}}{2-\beta}}\Bigr)
+(n_{k_0}^H)^{-1}\log(2\lambda_u),
\]
which is equivalent to
\[
\xi^{(n_{k_0}^H)^{-1}}
\Bigl(\lambda_s^{1+(n_{k_0}^H)^{-1}}\lambda_u^{\frac{7\alpha}{2}}\Bigr)^{(\alpha\beta)^{\lceil l/2\rceil}}
\ge
(2\lambda_u)^{(n_{k_0}^H)^{-1}}
\Bigl(\lambda_s\lambda_u^{\frac{9\beta-6+18(n_{k_0}^H)^{-1}}{2-\beta}}\Bigr)^{(\alpha\beta)^{\lfloor l/2\rfloor}}.
\]
Raising both sides to the power $n_{k_0}^H$, we obtain
\begin{equation}\label{eq:lem53step1}
\lambda_u^{-1}\xi
\Bigl(\lambda_s^{1+(n_{k_0}^H)^{-1}}\lambda_u^{\frac{7\alpha}{2}}\Bigr)^{n_{k_0}^H(\alpha\beta)^{\lceil l/2\rceil}}
\ge
2\Bigl(\lambda_s\lambda_u^{\frac{9\beta-6+18(n_{k_0}^H)^{-1}}{2-\beta}}\Bigr)^{n_{k_0}^H(\alpha\beta)^{\lfloor l/2\rfloor}}
=
2\epsilon_{k_0,l}.
\end{equation}

Since $k_0$ is even, the explicit formulas in {\bf OE} give
\[
n_{k_0+l}^H\ge n_{k_0}^H(\alpha\beta)^{\lfloor l/2\rfloor}-1
\qquad\text{and}\qquad
n_{k_0+l+1}^H\ge n_{k_0}^H(\alpha\beta)^{\lceil l/2\rceil}-1.
\]
Hence
\begin{equation}\label{eq:lem53step2}
\begin{split}
\lambda_u^{-n_{k_0+l}^H}
&\le
\Bigl(
\lambda_u^{-(\alpha\beta)^{\lfloor l/2\rfloor-\lceil l/2\rceil}
+(n_{k_0}^H)^{-1}(\alpha\beta)^{-\lceil l/2\rceil}}
\Bigr)^{n_{k_0}^H(\alpha\beta)^{\lceil l/2\rceil}},\\
\lambda_u^{-n_{k_0+l+1}^H}
&\le
\Bigl(
\lambda_u^{-1+(n_{k_0}^H)^{-1}(\alpha\beta)^{-\lceil l/2\rceil}}
\Bigr)^{n_{k_0}^H(\alpha\beta)^{\lceil l/2\rceil}}.
\end{split}
\end{equation}

From the construction \eqref{003} of $b_k$, it holds that
\[
b_k\le \lambda_u^{-6n_k^H}.
\]
Indeed, since $n_{k+i}^H\ge n_k^H$ for every $i\ge0$ by repeated use of \eqref{eq:20260314mono},
\[
\sum_{i=0}^{\infty}\frac{n_{k+i}^H}{2^i}
\ge
\sum_{i=0}^{\infty}\frac{n_k^H}{2^i}
=
2n_k^H.
\]
Therefore,
\[
\sqrt{b_{k_0+l+1}}\le \lambda_u^{-3n_{k_0+l+1}^H}.
\]
Using this together with \eqref{006}, \eqref{eq:lem53step3}, \eqref{eq:lem53step1}, and \eqref{eq:lem53step2}, we obtain
\refstepcounter{equation}\label{eq:lem53upper}
\begin{align*}
\epsilon_{k_0,l}b_{k_0+l}\lambda_u^{2n_{k_0+l}^H}
&=
\epsilon_{k_0,l}\lambda_u^{-n_{k_0+l}^H}\sqrt{b_{k_0+l+1}}
\\
&\le
\epsilon_{k_0,l}\lambda_u^{-n_{k_0+l}^H}\lambda_u^{-3n_{k_0+l+1}^H}
\\
&\le
\frac12\lambda_u^{-1}\xi
\Bigl(\lambda_s^{1+(n_{k_0}^H)^{-1}}\lambda_u^{\frac{7\alpha}{2}}\Bigr)^{n_{k_0}^H(\alpha\beta)^{\lceil l/2\rceil}}
\lambda_u^{-n_{k_0+l}^H}\lambda_u^{-3n_{k_0+l+1}^H}
\\
&=
\frac12\lambda_u^{-1}\xi
\Bigl(\lambda_s^{1+(n_{k_0}^H)^{-1}}\lambda_u^{\frac{7\alpha}{2}}\Bigr)^{n_{k_0}^H(\alpha\beta)^{\lceil l/2\rceil}}
\lambda_u^{-2n_{k_0+l+1}^H}\lambda_u^{-n_{k_0+l}^H}\lambda_u^{-n_{k_0+l+1}^H}
\\
&\le
\frac12\lambda_u^{-1}\xi
\Bigl(
\lambda_s^{1+(n_{k_0}^H)^{-1}}
\lambda_u^{\frac{7\alpha}{2}
-(\alpha\beta)^{\lfloor l/2\rfloor-\lceil l/2\rceil}
-2
+3(n_{k_0}^H)^{-1}(\alpha\beta)^{-\lceil l/2\rceil}}
\Bigr)^{n_{k_0}^H(\alpha\beta)^{\lceil l/2\rceil}}
\lambda_u^{-n_{k_0+l+1}^H}
\\
&<
\frac12\lambda_u^{-1}\xi
\Bigl(\lambda_s^{1+(n_{k_0}^H)^{-1}}\lambda_u^{\alpha}\Bigr)^{n_{k_0}^H(\alpha\beta)^{\lceil l/2\rceil}}
\lambda_u^{-n_{k_0+l+1}^H}
\\
&<
\xi\lambda_u^{-1}
\Bigl(\lambda_s^{1+(n_{k_0}^H)^{-1}}\lambda_u^{\alpha}\Bigr)^{n_{k_0}^H(\alpha\beta)^{\lceil l/2\rceil}}
\lambda_u^{-n_{k_0+l+1}^H}
\tag{\theequation}\raisetag{-1.2\baselineskip}
\end{align*}

We next show that
\begin{equation}\label{eq:lem53lower}
\lambda_u^{-1}
\Bigl(\lambda_s^{1+(n_{k_0}^H)^{-1}}\lambda_u^{\alpha}\Bigr)^{n_{k_0}^H(\alpha\beta)^{\lceil l/2\rceil}}
\lambda_u^{-n_{k_0+l+1}^H}
\le
\lambda_s^{n_{k_0+l}^H}.
\end{equation}

If $l=2p$, then the explicit formulas in {\bf OE} imply
\[
n_{k_0+2p}^H<(n_{k_0}^H+1)\alpha^p\beta^p
\qquad\text{and}\qquad
n_{k_0+2p+1}^H>n_{k_0}^H\alpha^{p+1}\beta^p-1.
\]
Hence
\[
\lambda_s^{n_{k_0+2p}^H}
>
\lambda_s^{(n_{k_0}^H+1)\alpha^p\beta^p},
\qquad
\lambda_u^{n_{k_0+2p+1}^H}
>
\lambda_u^{n_{k_0}^H\alpha^{p+1}\beta^p-1},
\]
and therefore
\[
\lambda_s^{n_{k_0+2p}^H}\lambda_u^{n_{k_0+2p+1}^H}
>
\lambda_u^{-1}
\Bigl(\lambda_s^{1+(n_{k_0}^H)^{-1}}\lambda_u^{\alpha}\Bigr)^{n_{k_0}^H\alpha^p\beta^p}
=
\lambda_u^{-1}
\Bigl(\lambda_s^{1+(n_{k_0}^H)^{-1}}\lambda_u^{\alpha}\Bigr)^{n_{k_0}^H(\alpha\beta)^p}.
\]

If $l=2p+1$, then the explicit formulas in {\bf OE} imply
\[
n_{k_0+2p+1}^H<(n_{k_0}^H+1)\alpha^{p+1}\beta^p
\]
and
\[
n_{k_0+2p+2}^H>n_{k_0}^H\alpha^{p+1}\beta^{p+1}-1
>n_{k_0}^H\alpha^{p+2}\beta^p-1.
\]
Hence
\[
\lambda_s^{n_{k_0+2p+1}^H}
>
\lambda_s^{(n_{k_0}^H+1)\alpha^{p+1}\beta^p},
\qquad
\lambda_u^{n_{k_0+2p+2}^H}
>
\lambda_u^{n_{k_0}^H\alpha^{p+2}\beta^p-1},
\]
so
\[
\lambda_s^{n_{k_0+2p+1}^H}\lambda_u^{n_{k_0+2p+2}^H}
>
\lambda_u^{-1}
\Bigl(\lambda_s^{1+(n_{k_0}^H)^{-1}}\lambda_u^{\alpha}\Bigr)^{n_{k_0}^H\alpha^{p+1}\beta^p}.
\]
Since
\[
\lambda_s^{1+(n_{k_0}^H)^{-1}}\lambda_u^{\alpha}
<
\lambda_s\lambda_u^{\alpha}
<
1,
\]
we further obtain
\[
\lambda_s^{n_{k_0+2p+1}^H}\lambda_u^{n_{k_0+2p+2}^H}
>
\lambda_u^{-1}
\Bigl(\lambda_s^{1+(n_{k_0}^H)^{-1}}\lambda_u^{\alpha}\Bigr)^{n_{k_0}^H\alpha^{p+1}\beta^{p+1}}.
\]

Therefore, in either case,
\[
\lambda_s^{n_{k_0+l}^H}
=
\Bigl(\lambda_s^{n_{k_0+l}^H}\lambda_u^{n_{k_0+l+1}^H}\Bigr)\lambda_u^{-n_{k_0+l+1}^H}
\ge
\lambda_u^{-1}
\Bigl(\lambda_s^{1+(n_{k_0}^H)^{-1}}\lambda_u^{\alpha}\Bigr)^{n_{k_0}^H(\alpha\beta)^{\lceil l/2\rceil}}
\lambda_u^{-n_{k_0+l+1}^H},
\]
which proves \eqref{eq:lem53lower}.

Combining \eqref{eq:lem53upper} and \eqref{eq:lem53lower}, we obtain \eqref{eq:lem53goal}. This completes the proof.
\end{proof}

In conclusion, by \eqref{eq:20251101}, \eqref{eq:1012}, \eqref{eq:20250927a} and Lemma~\ref{lem:5.3},
for any $l\ge 0$,
\begin{equation}\label{eq:0412}
    A_l=\begin{bmatrix}
        b_l^{11}\lambda_s^{n_{k_0+l}^H} & b_l^{12}\lambda_u^{n_{k_0+l}^H}\\
        b_l^{21}\lambda_s^{n_{k_0+l}^H} & b_l^{22}\lambda_s^{n_{k_0+l}^H}
    \end{bmatrix},
\end{equation}
where the $b_l^{ij}$ satisfy
\begin{equation}\label{eq:1004a}
    \max_{ij}|b_l^{ij}|\le\Lambda_1^{n_{k_0+l}^T},\quad
    |b_l^{22}|\le\xi\Lambda_1^{n_{k_0+l}^T},\quad
    \min\{|b_l^{12}|,|b_l^{21}|\}\ge\Lambda_1^{-n_{k_0+l}^T}.
\end{equation}

\section{Asymptotic behavior of the product of the Jacobian matrices}\label{s:6}

In this section, we investigate the asymptotic behavior of $A^{(m)}$.
Denote by $a_{ij}^{(m)}$ the $(i,j)$-entry of $A^{(m)}$ for $i,j\in\{1,2\}$.
Recall $C_{f,m}$, $\Lambda_{s}^{(m)}$ and $\Lambda_{u}^{(m)}$ in Section~\ref{ss:k0}, and set
\[
C_{ij}^m=
\begin{cases}
a^{(m)}_{ij}(\Lambda_{u}^{(m)})^{-1} \quad &(\text{$ij\in \{11,12\}$}),\\[2mm]
a^{(m)}_{ij}(\Lambda_{s}^{(m)})^{-1} \quad &(\text{$ij\in \{21,22\}$})
\end{cases}
\]
for $i,j \in \{1,2\}$.
Then
\[
A^{(m)}=
\begin{bmatrix}
C_{11}^m \Lambda_{u}^{(m)} & C_{12}^m \Lambda_{u}^{(m)}\\
C_{21}^m \Lambda_{s}^{(m)} & C_{22}^m \Lambda_{s}^{(m)}
\end{bmatrix}.
\]
The point is that the asymptotic behavior of $A^{(m)}$ is essentially determined by
$\Lambda_{u}^{(m)}$ and $\Lambda_{s}^{(m)}$, while the coefficients $C_{ij}^m$ remain under control.

\begin{lem}\label{lem:250404a}
For every odd integer $m\ge 1$, 
\begin{equation}\label{eq:1221-odd}
\left(1-\left(\frac{1}{3}+\frac{1}{3^2}+\cdots +\frac{1}{3^{m}}\right)\right)C_{f,m}^{-1}
\le |C_{ij}^m|
\le
\left(1+\frac{1}{3}+\frac{1}{3^2}+\cdots +\frac{1}{3^{m}}\right)C_{f,m}
\end{equation}
for each $ij\in\{12,21\}$, and
\begin{equation}\label{eq:1221-odd-ratio}
\max\left\{
\frac{|C_{11}^m|}{|C_{12}^m|},
\frac{|C_{22}^m|}{|C_{21}^m|}
\right\}
\le
\frac{1}{3}+\frac{1}{3^2}+\cdots +\frac{1}{3^{m}}.
\end{equation}

For every even integer $m\ge 1$, 
\begin{equation}\label{eq:1221-even}
\left(1-\left(\frac{1}{3}+\frac{1}{3^2}+\cdots +\frac{1}{3^{m}}\right)\right)C_{f,m}^{-1}
\le |C_{ij}^m|
\le
\left(1+\frac{1}{3}+\frac{1}{3^2}+\cdots +\frac{1}{3^{m}}\right)C_{f,m}
\end{equation}
for each $ij\in\{11,22\}$, and
\begin{equation}\label{eq:1221-even-ratio}
\max\left\{
\frac{|C_{12}^m|}{|C_{11}^m|},
\frac{|C_{21}^m|}{|C_{22}^m|}
\right\}
\le
\frac{1}{3}+\frac{1}{3^2}+\cdots +\frac{1}{3^{m}}.
\end{equation}

In particular, for every $m\ge 1$ and $i,j\in\{1,2\}$,
\[
|C_{ij}^m|<2C_{f,m}.
\]
\end{lem}

\begin{proof}
For $m=1$, we have
\[
A^{(1)}=A_0=
\begin{bmatrix}
b_0^{11}\lambda_s^{n_{k_0}^H} & b_0^{12}\lambda_u^{n_{k_0}^H}\\
b_0^{21}\lambda_s^{n_{k_0}^H} & b_0^{22}\lambda_s^{n_{k_0}^H}
\end{bmatrix},
\]
so
\[
C_{11}^1=b_0^{11}\Bigl(\frac{\lambda_s}{\lambda_u}\Bigr)^{n_{k_0}^H},
\qquad
C_{12}^1=b_0^{12},
\qquad
C_{21}^1=b_0^{21},
\qquad
C_{22}^1=b_0^{22}.
\]
By \eqref{eq:1004a},
\[
C_{f,1}^{-1}\le |C_{ij}^1|\le C_{f,1}\quad (ij\in\{12,21\}).
\]
Also,
\[
\frac{|C_{11}^1|}{|C_{12}^1|}
=
\Bigl(\frac{\lambda_s}{\lambda_u}\Bigr)^{n_{k_0}^H}
\left|\frac{b_0^{11}}{b_0^{12}}\right|
\le
\Lambda_1^{2n_{k_0}^T}
\Bigl(\frac{\lambda_s}{\lambda_u}\Bigr)^{n_{k_0}^H}
\le 10^{-4}
\]
by \eqref{eq:1004a} and \eqref{eq:20251018b}. Similarly,
\[
\frac{|C_{22}^1|}{|C_{21}^1|}
=
\left|\frac{b_0^{22}}{b_0^{21}}\right|
\le
\xi \Lambda_1^{2n_{k_0}^T}
\le 10^{-4}
\]
by \eqref{eq:1004a} and \eqref{eq:20251018b}. 
Therefore, \eqref{eq:1221-odd} and \eqref{eq:1221-odd-ratio} hold for $m=1$.

Assume now that the desired estimates hold for $m=l-1$ with an integer $l\ge 2$.
Since
\[
A^{(l)}=A_{l-1}A^{(l-1)},
\quad
\Lambda_u^{(l)}=\lambda_u^{n_{k_0+l-1}^H}\Lambda_s^{(l-1)},
\quad
\Lambda_s^{(l)}=\lambda_s^{n_{k_0+l-1}^H}\Lambda_u^{(l-1)},
\]
it follows from \eqref{eq:0412} that
\begin{align*}
&C_{11}^{l}
=b_{l-1}^{11}C_{11}^{l-1}\frac{\Lambda_s^{(l)}}{\Lambda_u^{(l)}}+b_{l-1}^{12}C_{21}^{l-1},\quad
C_{12}^{l}
=b_{l-1}^{11}C_{12}^{l-1}\frac{\Lambda_s^{(l)}}{\Lambda_u^{(l)}}+b_{l-1}^{12}C_{22}^{l-1},\\
&C_{21}^{l}
=b_{l-1}^{21}C_{11}^{l-1}+b_{l-1}^{22}C_{21}^{l-1}\frac{\Lambda_s^{(l-1)}}{\Lambda_u^{(l-1)}},\quad
C_{22}^{l}
=b_{l-1}^{21}C_{12}^{l-1}+b_{l-1}^{22}C_{22}^{l-1}\frac{\Lambda_s^{(l-1)}}{\Lambda_u^{(l-1)}}.
\end{align*}
Since
$
\frac{1}{3}+\frac{1}{3^2}+\cdots +\frac{1}{3^{l-1}}<\frac12,
$
the inductive hypothesis implies
\[
|C_{ij}^{\,l-1}|<\frac32\,C_{f,l-1}
\]
for all $i,j\in\{1,2\}$.
Indeed, if $l-1$ is odd, then
\[
|C_{ij}^{\,l-1}|
\le
\left(1+\frac{1}{3}+\cdots +\frac{1}{3^{l-1}}\right)C_{f,l-1}
<
\frac32\,C_{f,l-1}
\]
for $ij\in\{12,21\}$,
and
\[
|C_{11}^{\,l-1}|
\le
\left(\frac{1}{3}+\cdots +\frac{1}{3^{l-1}}\right)|C_{12}^{\,l-1}|
<
\frac12\cdot\frac32\,C_{f,l-1}
<
\frac32\,C_{f,l-1},
\]
\[
|C_{22}^{\,l-1}|
\le
\left(\frac{1}{3}+\cdots +\frac{1}{3^{l-1}}\right)|C_{21}^{\,l-1}|
<
\frac12\cdot\frac32\,C_{f,l-1}
<
\frac32\,C_{f,l-1}.
\]
If $l-1$ is even, the same argument with the roles of
$(C_{11}^{\,l-1},C_{22}^{\,l-1})$ and $(C_{12}^{\,l-1},C_{21}^{\,l-1})$
interchanged gives again the claimed estimate.

Hence \eqref{eq:1004a} and \eqref{eq:k0decay} imply
\[
\left|b_{l-1}^{\mu\nu}C_{ij}^{\,l-1}\frac{\Lambda_s^{(l)}}{\Lambda_u^{(l)}}\right|
\le
\Lambda_1^{n_{k_0+l-1}^T}\cdot \frac32\,C_{f,l-1}\cdot
\frac{1}{12\cdot 3^l\,C_{f,l}^2}
=
\frac{1}{8\cdot 3^l}\,C_{f,l}^{-1}
\]
for all relevant indices, because
$
C_{f,l}=\Lambda_1^{n_{k_0+l-1}^T}C_{f,l-1}.
$
Moreover, \eqref{eq:20251018b} implies $\xi<1/3$, and therefore
\begin{align*}
\left|b_{l-1}^{22}C_{ij}^{\,l-1}\frac{\Lambda_s^{(l-1)}}{\Lambda_u^{(l-1)}}\right|
&\le
\xi\Lambda_1^{n_{k_0+l-1}^T}\cdot \frac32\,C_{f,l-1}\cdot
\frac{1}{12\cdot 3^{l-1}\,C_{f,l-1}^2\Lambda_1^{2n_{k_0+l-1}^T}}\\
&=
\frac{\xi}{8\cdot 3^{l-1}\Lambda_1^{n_{k_0+l-1}^T}C_{f,l-1}}\\
&=
\frac{\xi}{8\cdot 3^{l-1}}\,C_{f,l}^{-1}
\le
\frac{1}{8\cdot 3^l}\,C_{f,l}^{-1}
\end{align*}
for all relevant indices.

Suppose first that $l$ is odd. Then $l-1$ is even, so it follows from the above observations that
\begin{align*}
|C_{12}^{l}|
&\le
\frac{1}{8\cdot 3^l}C_{f,l}^{-1}
+\Lambda_1^{n_{k_0+l-1}^T}
\left(1+\frac{1}{3}+\cdots +\frac{1}{3^{l-1}}\right)C_{f,l-1}\\
&\le
\left(1+\frac{1}{3}+\cdots +\frac{1}{3^{l}}\right)C_{f,l},
\end{align*}
because $C_{f,l}^{-1}\le C_{f,l}$ and $C_{f,l}=\Lambda_1^{n_{k_0+l-1}^T}C_{f,l-1}$. Also,
\begin{align*}
|C_{12}^{l}|
&\ge
-\frac{1}{8\cdot 3^l}C_{f,l}^{-1}
+\Lambda_1^{-n_{k_0+l-1}^T}
\left(1-\left(\frac{1}{3}+\cdots +\frac{1}{3^{l-1}}\right)\right)C_{f,l-1}^{-1}\\
&\ge
\left(1-\left(\frac{1}{3}+\cdots +\frac{1}{3^{l}}\right)\right)C_{f,l}^{-1}.
\end{align*}
Exactly the same argument gives
\[
\left(1-\left(\frac{1}{3}+\cdots +\frac{1}{3^{l}}\right)\right)C_{f,l}^{-1}
\le |C_{21}^{l}|
\le
\left(1+\frac{1}{3}+\cdots +\frac{1}{3^{l}}\right)C_{f,l}.
\]
Hence \eqref{eq:1221-odd} holds for $m=l$.

We next prove the ratio estimate \eqref{eq:1221-odd-ratio}.
Let
\[
L_a=b_{l-1}^{11}C_{11}^{\,l-1}\frac{\Lambda_s^{(l)}}{\Lambda_u^{(l)}},
\quad
R_a=b_{l-1}^{12}C_{21}^{\,l-1},\quad
L_b=b_{l-1}^{11}C_{12}^{\,l-1}\frac{\Lambda_s^{(l)}}{\Lambda_u^{(l)}},
\quad
R_b=b_{l-1}^{12}C_{22}^{\,l-1}.
\]
Then
\[
\frac{|R_a|}{|R_b|}
=
\frac{|C_{21}^{\,l-1}|}{|C_{22}^{\,l-1}|}
\le
\frac{1}{3}+\frac{1}{3^2}+\cdots +\frac{1}{3^{l-1}}
\]
by \eqref{eq:1221-even-ratio} in the inductive hypothesis. On the other hand,
\[
|L_a|,\ |L_b|
\le
\frac{1}{8\cdot 3^l}C_{f,l}^{-1},
\qquad
|R_b|
\ge
\left(1-\left(\frac{1}{3}+\cdots +\frac{1}{3^{l-1}}\right)\right)C_{f,l}^{-1}
>
\frac12\,C_{f,l}^{-1},
\]
so
\[
\frac{|L_a|}{|R_b|},\ \frac{|L_b|}{|R_b|}
\le
\frac{1}{4\cdot 3^l}.
\]
Therefore
\begin{align*}
\frac{|C_{11}^{l}|}{|C_{12}^{l}|}
=
\frac{|L_a+R_a|}{|L_b+R_b|}
\le
\frac{|L_a|+|R_a|}{|R_b|-|L_b|}
\le
\frac{\frac{1}{4\cdot 3^l}
+\left(\frac{1}{3}+\cdots +\frac{1}{3^{l-1}}\right)}
{1-\frac{1}{4\cdot 3^l}}
\le
\frac{1}{3}+\frac{1}{3^2}+\cdots +\frac{1}{3^{l}}
\end{align*}
since
$
\frac{1}{4\cdot 3^l}
(
1+\frac{1}{3}+\cdots +\frac{1}{3^{l-1}}+\frac{1}{3^l}
)
\le \frac{1}{3^l}.
$
The proof of
$
\frac{|C_{22}^{l}|}{|C_{21}^{l}|}
\le
\frac{1}{3}+\frac{1}{3^2}+\cdots +\frac{1}{3^{l}}
$
is entirely similar. Thus \eqref{eq:1221-odd-ratio}
holds for $m=l$.

Suppose next that $l$ is even. Then noticing that $l-1$ is odd, 
and repeating the same argument, with the roles of
$(C_{11}^{\,l-1},C_{22}^{\,l-1})$ and $(C_{12}^{\,l-1},C_{21}^{\,l-1})$
interchanged, we obtain
\eqref{eq:1221-even} and \eqref{eq:1221-even-ratio}
for $m=l$.
This completes the induction and the proof.
\end{proof}

Fix $C_0\in (1,2)$ and let
\[
V:=
\left\{
(c^s,c^u)\in \mathbb R^2\setminus\{0\}
\ \middle|\ 
C_0^{-1}< \frac{|c^s|}{|c^u|}< C_0
\right\}.
\]
Then
\[
\min\left\{|c^s|-\frac12|c^u|,\ |c^u|-\frac12|c^s|\right\}
\ge \frac{2-C_0}{2C_0}|c^s|
\qquad ((c^s,c^u)\in V).
\]
Now fix $x\in U_{k_0,0}$ and $c=(c_0^s,c_0^u)\in V$.
Set
\[
c'=(c_m^s,c_m^u):=A^{(m)}c,
\qquad
v:=c_0^s v^s+c_0^u v^u,
\qquad
v':=c_m^s v^s+c_m^u v^u.
\]
Then $Df^{N_m}(x)v=v'$ because both $x$ and $f^{N_m}(x)$ are in $K$.
For an odd integer $m\ge 1$, \eqref{eq:1221-odd} and \eqref{eq:1221-odd-ratio} in Lemma~\ref{lem:250404a} give
\[
|c_m^s|
=\Lambda_u^{(m)}|C_{11}^m c_0^s+C_{12}^m c_0^u|
\ge \Lambda_u^{(m)}|C_{12}^m|\left(|c_0^u|-\frac12|c_0^s|\right),
\]
and hence
\[
|c_m^s|
\ge
\frac{13(2-C_0)}{54C_0}C_{f,m}^{-1}\Lambda_u^{(m)}|c_0^s|.
\]
The same inequality for even integers $m\ge 1$ follows from \eqref{eq:1221-even} and \eqref{eq:1221-even-ratio}.
On the other hand, Lemma~\ref{lem:250404a} also gives
\[
|c_m^u|
=\Lambda_s^{(m)}|C_{21}^m c_0^s+C_{22}^m c_0^u|
\le 2(1+C_0)C_{f,m}\Lambda_s^{(m)}|c_0^s|.
\]
Therefore
\[
\|Df^{N_m}(x)v\|
=\|v'\|
\ge
\frac{13(2-C_0)}{54C_0}C_{f,m}^{-1}\Lambda_u^{(m)}|c_0^s|
-2(1+C_0)C_{f,m}\Lambda_s^{(m)}|c_0^s|.
\]
On the other hand, the ratio $\Lambda_s^{(m)}/\Lambda_u^{(m)}$ is bounded by $1$, so we have
\[
\|Df^{N_m}(x)v\|
\le
2(1+C_0)C_{f,m}(\Lambda_u^{(m)}+\Lambda_s^{(m)})|c_0^s|\le 4(1+C_0) C_{f,m}\Lambda_u^{(m)}|c_0^s|.
\]
Thus, by \eqref{eq:k0decay},
for all sufficiently large $m$ we have
\begin{equation}\label{eq:0308d}
\frac{13(2-C_0)}{2\cdot 54C_0}C_{f,m}^{-1}\Lambda_u^{(m)}|c_0^s|
\le
\|Df^{N_m}(x)v\|
\le
4(1+C_0) C_{f,m}\Lambda_u^{(m)}|c_0^s|.
\end{equation}

\begin{lem}\label{lem:250404b}
It holds that
\[
\lim_{m\to\infty}\frac{\log\Lambda_u^{(2m)}}{N_{2m}}
=
\frac{\alpha\log\lambda_u+\log\lambda_s}{1+\alpha}
<
\frac{\beta\log\lambda_u+\log\lambda_s}{1+\beta}
=
\lim_{m\to\infty}\frac{\log\Lambda_u^{(2m-1)}}{N_{2m-1}}.
\]
\end{lem}

\begin{proof}
Since $k_0$ was chosen to be even, we write $k_0=2q$.
By {\bf OE},
\[
n_{k_0+2p}^H
=
\lfloor n_0^H\alpha^{q+p}\beta^{q+p}\rfloor
=
c(\alpha\beta)^p+O(1),
\]
\[
n_{k_0+2p+1}^H
=
\lfloor n_0^H\alpha^{q+p+1}\beta^{q+p}\rfloor
=
c\alpha(\alpha\beta)^p+O(1)
\qquad (p\to\infty),
\]
where $c:=n_0^H\alpha^q\beta^q>0$.
Hence
\[
\sum_{j=0}^{m-1}n_{k_0+2j}^H
=
c\sum_{j=0}^{m-1}(\alpha\beta)^j+O(m),
\qquad
\sum_{j=0}^{m-1}n_{k_0+2j+1}^H
=
c\alpha\sum_{j=0}^{m-1}(\alpha\beta)^j+O(m).
\]
Therefore
\[
\log\Lambda_u^{(2m)}
=
c\alpha \log\lambda_u \sum_{j=0}^{m-1}(\alpha\beta)^j
+
c\log\lambda_s\sum_{j=0}^{m-1}(\alpha\beta)^j
+O(m),
\]
while
\[
\sum_{j=0}^{2m-1}n_{k_0+j}^H
=
c\alpha\sum_{j=0}^{m-1}(\alpha\beta)^j
+
c\sum_{j=0}^{m-1}(\alpha\beta)^j
+O(m).
\]
Since $\frac{\sum_{j=0}^{2m-1}n_{k_0+j}^H}{N_{2m}}\to 1$ by {\bf DH},
it follows that
\[
\lim_{m\to\infty}\frac{\log\Lambda_u^{(2m)}}{N_{2m}}
=
\lim_{m\to\infty}\frac{\log\Lambda_u^{(2m)}}{\sum_{j=0}^{2m-1}n_{k_0+j}^H}
=
\frac{\alpha\log\lambda_u+\log\lambda_s}{1+\alpha}.
\]

Next, notice that
\[
\sum_{j=0}^{m-2}n_{k_0+2j+1}^H
=
c\alpha\sum_{j=0}^{m-2}(\alpha\beta)^j+O(m),
\]
so
\[
\log\Lambda_u^{(2m-1)}
=
c\log\lambda_u\sum_{j=0}^{m-1}(\alpha\beta)^j
+
c\alpha\log\lambda_s\sum_{j=0}^{m-2}(\alpha\beta)^j
+O(m),
\]
while
\[
\sum_{j=0}^{2m-2}n_{k_0+j}^H
=
c\sum_{j=0}^{m-1}(\alpha\beta)^j
+
c\alpha\sum_{j=0}^{m-2}(\alpha\beta)^j
+O(m).
\]
Since $\frac{\sum_{j=0}^{2m-2}n_{k_0+j}^H}{N_{2m-1}}\to 1$ by {\bf DH}, and
\[
\frac{\sum_{j=0}^{m-2}(\alpha\beta)^j}{\sum_{j=0}^{m-1}(\alpha\beta)^j}\to \frac{1}{\alpha\beta}
\qquad (m\to\infty),
\]
we obtain
\[
\lim_{m\to\infty}\frac{\log\Lambda_u^{(2m-1)}}{N_{2m-1}}
=
\frac{\log\lambda_u+\beta^{-1}\log\lambda_s}{1+\beta^{-1}}
=
\frac{\beta\log\lambda_u+\log\lambda_s}{1+\beta}.
\]
\end{proof}

By {\bf DH},
\[
\frac{\log C_{f,m}}{N_m}
=
\frac{(n_{k_0}^T+\cdots+n_{k_0+m-1}^T)\log\Lambda_1}{n_{k_0}+\cdots+n_{k_0+m-1}}
\to 0
\]
as $m\to\infty$.
Combining this with \eqref{eq:0308d} and Lemma~\ref{lem:250404b}, we conclude that
\[
\lim_{m\to\infty}\frac{1}{N_{2m}}\log \|Df^{N_{2m}}(x)v\|
<
\lim_{m\to\infty}\frac{1}{N_{2m-1}}\log \|Df^{N_{2m-1}}(x)v\|.
\]
This completes the proof of Theorem~\ref{thm:main2}.
\begin{remark}
As can be seen from the above proof, the oscillation of the Lyapunov exponents in Theorem~\ref{thm:main2} is uniform over $U$ and $V$.
That is, under the identification of $V$ with the corresponding cone in $T_xM$ via the basis $(v^s,v^u)$, for all $x\in U$ and $v\in V$, it holds that
\begin{align*}
\liminf_{n\to\infty}\frac{1}{n}\log\|Df^n(x)v\| \le &\frac{\alpha\log\lambda_u+\log\lambda_s}{1+\alpha}\\
<&
\frac{\beta\log\lambda_u+\log\lambda_s}{1+\beta} \le \limsup_{n\to\infty}\frac{1}{n}\log\|Df^n(x)v\|.
\end{align*}
\end{remark}

\section*{Acknowledgement}
S.~Kiriki was supported by JSPS KAKENHI Grant Number 21K03332.
Y.~Nakano was supported by JSPS KAKENHI Grant  Number 23K03188 and JST PRESTO Grant Number JPMJPR25K8.
T.~Soma was supported by JSPS KAKENHI Grant Number 22K03342.

\bibliographystyle{abbrv}
\bibliography{ref}

\end{document}